\documentclass[12pt]{amsart}
\usepackage[none]{hyphenat}
\usepackage[utf8]{inputenc}
\usepackage[T1]{fontenc}
\usepackage{amsfonts}
\usepackage[leqno]{amsmath}
\usepackage{amssymb, comment, float}
\usepackage[dvipsnames]{xcolor}
\usepackage[foot]{amsaddr}
\usepackage[shortlabels]{enumitem}
\usepackage{mathdots}
\usepackage[footskip=0.5in,  headheight = 0.5in, top=1.25in, bottom=1.25in,  right=1in,  left=1in]{geometry}
\usepackage{hyperref}
\hypersetup{
colorlinks,
linkcolor=black,
urlcolor=Blue,
citecolor=black,}
\usepackage[center]{caption}
\usepackage{eufrak}
\usepackage{marvosym}     
\usepackage{latexsym}
\usepackage[nodayofweek]{datetime}
\usepackage{tikz}
\usepackage{subfig}
\usepackage{thm-restate}

\newtheorem{theorem}{Theorem}[section]
\newtheorem{lemma}[theorem]{Lemma}

\newtheorem{question}[theorem]{Question}
\newtheorem{observation}[theorem]{Observation}

\DeclareMathOperator{\tw}{tw}
\DeclareMathOperator{\pw}{pw}

\def\dd{\hbox{-}}   

\newcommand{\mf}{\mathfrak}
\newcommand{\mca}{\mathcal}
\newcommand{\poi}{\mathbb{N}} 

\newcommand{\pre}{\preccurlyeq}

\newcounter{tbox}
\newcommand{\sta}[1]{\medskip\medskip\refstepcounter{tbox}\noindent{\parbox{\textwidth}{(\thetbox) \emph{#1}}}\vspace*{0.3cm}}
\newcommand{\mylongtitle}[1]{%
  \ifodd\value{page}%
    \protect\parbox{0.97\linewidth}{#1}\hfill%
  \else%
    \hfill\protect\parbox{0.97\linewidth}{#1}%
  \fi%
}

\title[Induced subgraphs and tree decompositions XVI.]{Induced subgraphs and tree decompositions\\
XVI. Complete bipartite induced minors}

\author{Maria Chudnovsky$^{\dagger \ast}$}
\author{Sepehr Hajebi$^{\mathsection}$}
\author{Sophie Spirkl$^{\mathsection \parallel}$}

\thanks{$^{\dagger}$ Princeton University, Princeton, NJ, USA}
\thanks{$^{\mathsection}$ Department of Combinatorics and Optimization, University of Waterloo, Waterloo, Ontario, Canada}
\thanks{$^{\ast}$ Supported by  NSF-EPSRC Grant DMS-2120644, AFOSR grant FA9550-22-1-0083 and NSF Grant DMS-2348219.} 
\thanks{$^{\parallel}$ We acknowledge the support of the Natural Sciences and Engineering Research Council of Canada (NSERC), [funding reference number RGPIN-2020-03912].
Cette recherche a \'et\'e financ\'ee par le Conseil de recherches en sciences naturelles et en g\'enie du Canada (CRSNG), [num\'ero de r\'ef\'erence RGPIN-2020-03912]. This project was funded in part by the Government of Ontario. This research was conducted while Spirkl was an Alfred P. Sloan Fellow.}
\date {\today}
\begin{document}
\maketitle
\sloppy

\begin{abstract}
    We prove that for every graph $G$ with a sufficiently large complete bipartite induced minor, either $G$ has an induced minor isomorphic to a large wall, or $G$ contains a large \textit{constellation}; that is, a complete bipartite induced minor model such that on one side of the bipartition, each branch set is a singleton, and on the other side, each branch set induces a path.
    
    We further refine this theorem by characterizing the unavoidable induced subgraphs of large constellations as two types of highly structured constellations. These results will be key ingredients in several forthcoming papers of this series.

\end{abstract}

\section{Introduction}
The set of all positive integers is denoted by $\poi$. For an integer $k$, we write $\poi_k$ for the set of all positive integers not greater than $k$ (so $\poi_k=\varnothing$ if and only if $k\leq 0$). Graphs in this paper have finite vertex sets, no loops and no parallel edges. Given a graph $G=(V(G),E(G))$ and $X\subseteq V(G)$, we use $X$ to denote both the set $X$ and the subgraph $G[X]$ of $G$ induced by $X$. For a set $\mca{X}$ of subsets of $V(G)$, we write $V(\mca{X})=\bigcup_{X\in \mca{X}}X$. We say that $G$ is \textit{$H$-free}, for another graph $H$, if $G$ has no induced subgraph isomorphic to $H$.
\medskip

This series of papers studies the interplay between induced subgraphs and \textit{treewidth} (where the treewidth of a graph $G$ is denoted by $\tw(G)$; see \cite{diestel} for a definition).

Specifically,  one of the main goals is to answer the following question:

\begin{question}\label{q:GT}
    What are the unavoidable induced subgraphs of graphs with large treewidth?
\end{question}

The answer is known if we replace ``induced subgraphs'' by ``subgraphs'' or ``minors'' as shown by Robertson and Seymour \cite{GMV} (where $W_{r \times r}$ denotes the $r$-by-$r$ hexagonal grid, also known as the $r$-by-$r$ \textit{wall};
see Figure~\ref{fig:basic}):

\begin{theorem}[Robertson and Seymour \cite{GMV}]\label{thm:wallminor}
For every integer $r\in \poi$,
there is a constant $f_{\ref{thm:wallminor}}=f_{\ref{thm:wallminor}}(r)\in \poi$ such that every graph $G$ with $\tw(G) \geq f_{\ref{thm:wallminor}}$ has a subgraph isomorphic to a subdivision of $W_{r\times r}$.
\end{theorem}

  \begin{figure}[t!]
        \centering
\includegraphics[scale=0.6]{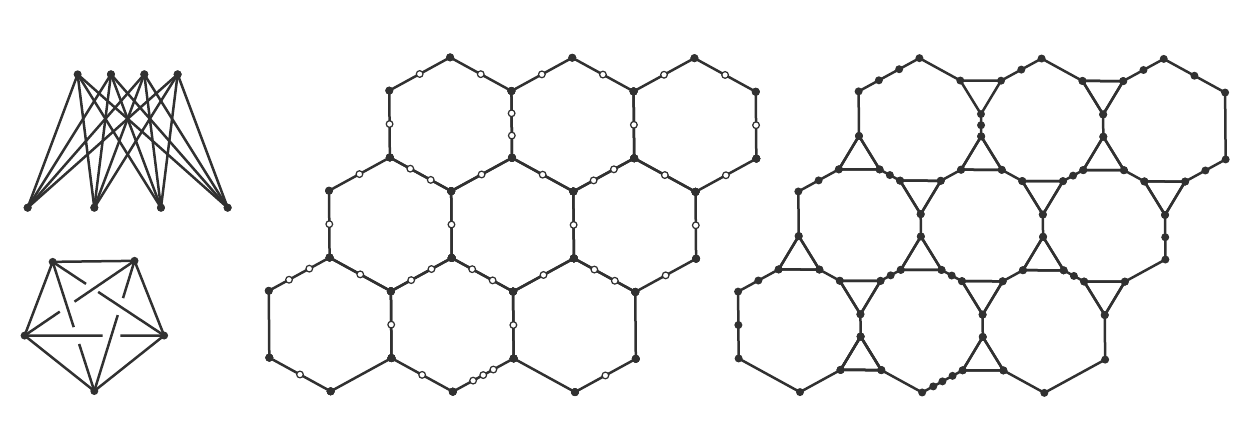}
        \caption{The basic obstructions of treewidth 4: $K_{4,4}$ (top left), $K_5$ (bottom left), a subdivision of the $4$-by-$4$ wall (middle) and the line graph of a subdivision of the $4$-by-$4$ wall (right).}
        \label{fig:basic}
    \end{figure}   

The answer to Question \ref{q:GT} is necessarily more complicated, as it needs to include all \emph{basic obstructions}: complete graphs, complete bipartite graphs, subdivided walls, and the line graphs of subdivided walls (see Figure~\ref{fig:basic}). Moreover, basic obstructions do not form a complete answer to Question \ref{q:GT}, as shown by several examples of ``non-basic'' obstructions such as the Pohoata-Davies graphs \cite{davies, pohoata2014unavoidable}, ``occultations'' \cite{tw9,deathstar}, and ``layered wheels'' \cite{sintiari2021theta}.

It appears that the key to answering Question~\ref{q:GT} is to understand the unavoidable induced subgraphs of graphs with large dense minors. In all non-basic obstructions, large treewidth is caused by the presence of a large complete minor (and that is for a good reason: it has been shown several times \cite{aboulker, seymourtalk, wollananswers} that every graph with large enough treewidth and no induced subgraph isomorphic to a basic obstruction of large treewidth has a large complete minor). In fact, the first two non-basic obstructions mentioned above, the Pohoata-Davies graphs and the occultations (see Figure~\ref{fig:pdo}), have large complete bipartite induced minors (while the third, layered wheels, do not \cite{nicolassuggests}. This also follows directly from our main result, Theorem~\ref{thm:motherKtt}.)

\begin{figure}[t!]
    \centering
    \includegraphics[scale=0.7]{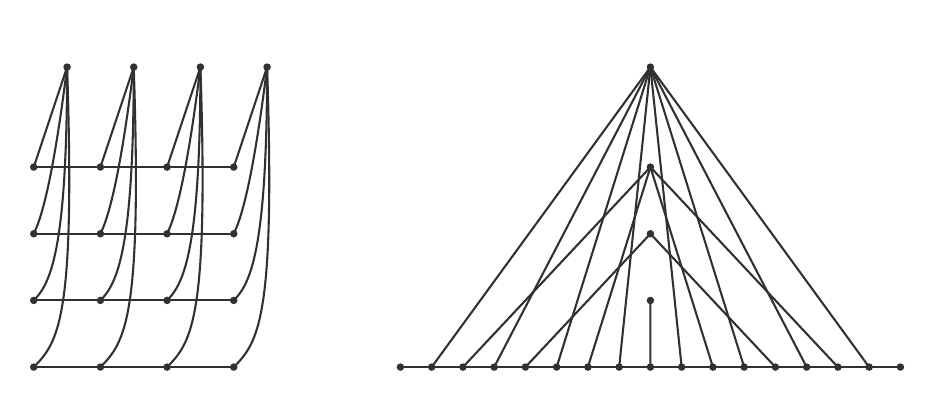}
    \caption{A Pohoata-Davies graph (left) and an occultation (right), both of treewidth $4$.}
    \label{fig:pdo}
\end{figure}

It is therefore natural to ask for a characterization of the unavoidable induced subgraphs of graphs with large complete bipartite induced minors (at least in the absence of basic obstructions of large treewidth). Our main result is exactly this characterization. As it turns out, the Pohoata-Davies graphs and the occultations are essentially the only (non-basic) outcomes.

Let us make all this precise, beginning with a few definitions. Let $G$ be a graph. We say that $X, Y\subseteq V(G)$ are \textit{anticomplete in $G$} if $X\cap Y=\varnothing$ and there is no edge in $G$ with an end in $X$ and an end in $Y$ (if $X=\{x\}$ is a singleton, then we also say \textit{$x$ is anticomplete to $Y$ in $G$}). Given a graph $H$, we say \textit{$G$ has an induced minor isomorphic to $H$} if there are pairwise disjoint connected induced subgraphs $(G_v:v\in V(H))$ of $G$, which we call the \textit{branch sets}, such that for all distinct $u,v\in V(G)$, we have $uv\in E(H)$ if and only if $G_u, G_v\subseteq V(G)$ are not anticomplete in $G$. Both the Pohoata-Davies graphs and the occultations contain complete bipartite induced minors of a special type: On one side of the bipartition, each branch set is a singleton, and on the other side, each branch set induces a path. We call such graphs \emph{constellations}, and analyzing the structure of constellations was one of the main tools in many of our earlier papers \cite{tw9, tw11, tw12, tw13, chordalehf}. Formally, a \textit{constellation} is a graph $\mf{c}$ in which there is a stable set $S_{\mf{c}}$ such that every component of $\mf{c}\setminus S_{\mf{c}}$ is a path, and each vertex $x\in S_{\mf{c}}$ has at least one neighbor in each component of $\mf{c}\setminus S_{\mf{c}}$. We denote by $\mca{L}_{\mf{c}}$ the set of all components $\mf{c}\setminus S_{\mf{c}}$ (each of which is a path), and denote the constellation $\mf{c}$ by the pair $(S_{\mf{c}},\mca{L}_{\mf{c}})$. For $l,s\in \poi$, by an \textit{$(s,l)$-constellation} we mean a constellation $\mf{c}$ with $|S_{\mf{c}}|=s$ and $|\mca{L}_{\mf{c}}|=l$ (see Figure~\ref{fig:constellation}). Given a graph $G$, by an \textit{$(s,l)$-constellation in $G$} we mean an induced subgraph of $G$ which is an $(s,l)$-constellation.

\begin{figure}[t!]
    \centering
    \includegraphics[scale=0.6]{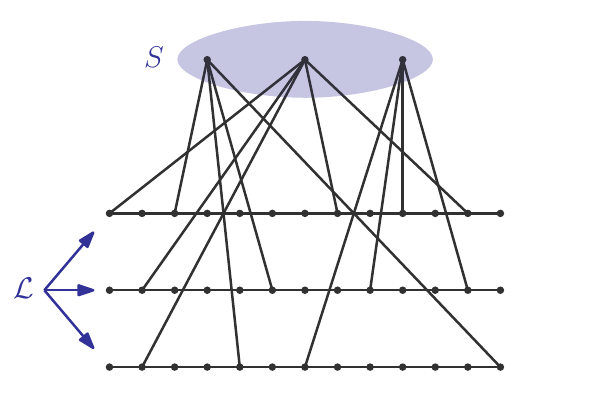}
    \caption{A $(3,3)$-constellation.}
    \label{fig:constellation}
\end{figure}
\medskip

Note that an $(s, l)$-constellation has treewidth at least $\min\{s, l\}$, because it has an induced minor isomorphic to the complete bipartite graph $K_{s,l}$ (obtained by contracting all edges of the paths in $\mathcal{L}_{\mf{c}}$). Our first result shows that, in the absence of basic obstructions, the converse is also true: 

\begin{theorem}\label{thm:main_ind_minor}
    For all $l,r,s\in \poi$, there is a constant $f_{\ref{thm:main_ind_minor}}=f_{\ref{thm:main_ind_minor}}(l,r,s)\in \poi$ such that if $G$ is a graph with an induced minor isomorphic to $K_{f_{\ref{thm:main_ind_minor}},f_{\ref{thm:main_ind_minor}}}$, then one of the following holds.
    \begin{enumerate}[{\rm (a)}]
        \item\label{thm:main_ind_minor_a} There is an induced minor of $G$ isomorphic to $W_{r\times r}$.
        \item\label{thm:main_ind_minor_a} There is an $(s,l)$-constellation in $G$.
    \end{enumerate}
\end{theorem}

It remains to describe the unavoidable induced subgraphs of large constellations, which we will do in our second result. The exact statement will be given as Theorem~\ref{thm:mainconstellation}, which roughly says the following: for all $l,r,s\in \poi$, given a sufficiently large constellation $\mf{c}$, either a basic obstruction of treewidth $r$ is present in $\mf{c}$ as an induced subgraph, or there is an $(s,l)$-constellation $\mf{b}$ in $\mf{c}$ along with a linear order $x_1, \dots, x_s$ of the vertices in $S_{\mf{b}}$ such that one of the following holds: 

\begin{itemize}
    \item For all $1\leq i < j\leq s$ and every path $P$ from $x_i$ to $x_j$ with interior in $V(\mathcal{L})$, every vertex $x_k$ with $j<k\leq s$ has a neighbor in $P$; or
    \item There is a constant $c = c(r)$ (actually, $c(r)=2r^2$ works) such that for all $1\leq i < k\leq s$ and every path $P$ from $x_i$ to $x_k$ with interior in $V(\mathcal{L})$, we have
    $$|\{j \in \{i+1, \dots, k-1\} : N(x_j) \cap V(P) = \varnothing\}|<c.$$
\end{itemize}
In particular, the first of these is a slight generalization of occultations. The second is a relatively substantial generalization of the Pohoata-Davies graphs, but we will show in Section~\ref{sec:zig} that this outcome cannot be replaced with (straightforward variants of) the Pohoata-Davies graphs. We will also prove in Section~\ref{sec:necessary} that both outcomes are necessary.

\section{Unavoidable constellations and the main result}

Here we give the exact statement of our second result, Theorem~\ref{thm:mainconstellation}. Then we state our main result, Theorem~\ref{thm:motherKtt}, as a direct combination of Theorems~\ref{thm:main_ind_minor} and \ref{thm:mainconstellation}.
\medskip

We need several definitions. Let $P$ be a graph which is a path. Then we write $P=p_1\dd \cdots\dd p_k$ to mean $V(P)=\{p_1,\ldots,p_k\}$ for $k\in \poi$, and $E(P)=\{p_ip_{i+1}:i\in \poi_{k-1}\}$. We call $p_1,p_k$ the \textit{ends} of $P$, and we call $P\setminus \{p_1,p_k\}$ the interior of $P$, denoted $P^*$. Given a graph $G$, by a \textit{path in $G$} we mean an induced subgraph of $G$ that is a path.
 
Given a graph $G$, a \emph{subdivision of $G$} is a graph obtained from $G$ by replacing each edge $e = uv$ of $G$ by a path $P_e$ of length at least 1 from $u$ to $v$ such that the interiors of the paths are pairwise disjoint and anticomplete. For $d \in \mathbb{N}$, this is a \emph{$d$-subdivision} (\emph{$(\leq d)$-subdivision, $(\geq d)$-subdivision}) if each path $P_e$ has length exactly $d+1$ (at most $d+1$, at least $d+1$) for all $e \in E(G)$. A subdivision is \emph{proper} if it is a $(\geq 1)$-subdivision.
For a set $X$, a linear order $\pre$ on $X$, and $x,y\in X$, we write $x\prec y$ to mean $x\pre y$ and $x$ and $y$ are distinct. For an element $x\in X$ and a subset $Y\subseteq X$, we write $x\prec Y$ to mean $x\prec y$ for every $y\in Y$. Similarly, we write $Y\prec x$ to mean $y\prec x$ for every $y\in Y$.
\begin{figure}[t!]
    \centering
    \includegraphics[scale=0.6]{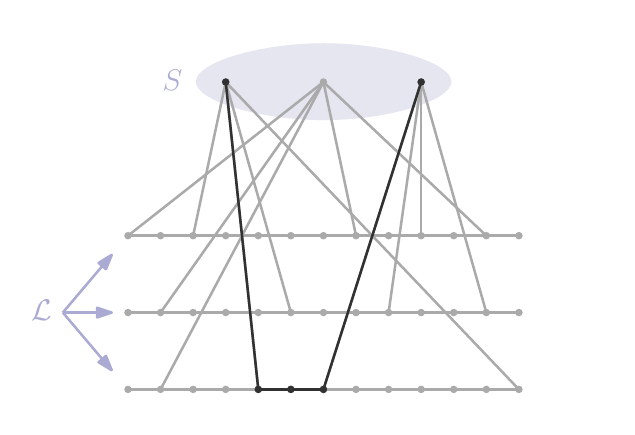}
    \caption{A $\mf{c}$-route of length four in the $(3,3)$-constellation $\mf{c}$ from Figure~\ref{fig:constellation}, which is $2$-ample.}
    \label{fig:route}
\end{figure}
\medskip

Let $\mf{c} = (S_{\mf{c}}, \mathcal{L}_{\mf{c}})$ be a constellation. By a \textit{$\mf{c}$-route} we mean a path $R$ in $\mf{c}$ with ends in $S_{\mf{c}}$ and with $R^*\subseteq V(\mca{L}_{\mf{c}})$, or equivalently, with $R^*\subseteq V(L)$ for some $L\in \mca{L}_{\mf{c}}$. For $d\in \poi$, we say $\mf{c}$ is \textit{$d$-ample} if there is no $\mf{c}$-route of length at most $d+1$. We also say $\mf{c}$ is \textit{ample} if $\mf{c}$ is $1$-ample (that is, no two vertices in $S_{\mf{c}}$ have a common neighbor in $V(\mca{L}_{\mf{c}})$; see Figure~\ref{fig:route}).

We say a constellation $\mf{c}$ is \textit{interrupted} if there is a linear order $\pre$ on $S_{\mf{c}}$ such that for all $x,y,z\in S_{\mf{c}}$ with $x\prec y\prec z$ and every $\mf{c}$-route $R$ from $x$ to $y$, the vertex $z$ has a neighbor in $R$ (see Figure~\ref{fig:interrupted}). We remark that occultations as in \cite{tw9,deathstar} are a special case of this (compare Figure~\ref{fig:pdo} (right) and Figure~\ref{fig:interrupted}).
\begin{figure}[t!]
    \centering
    \includegraphics[scale=0.6]{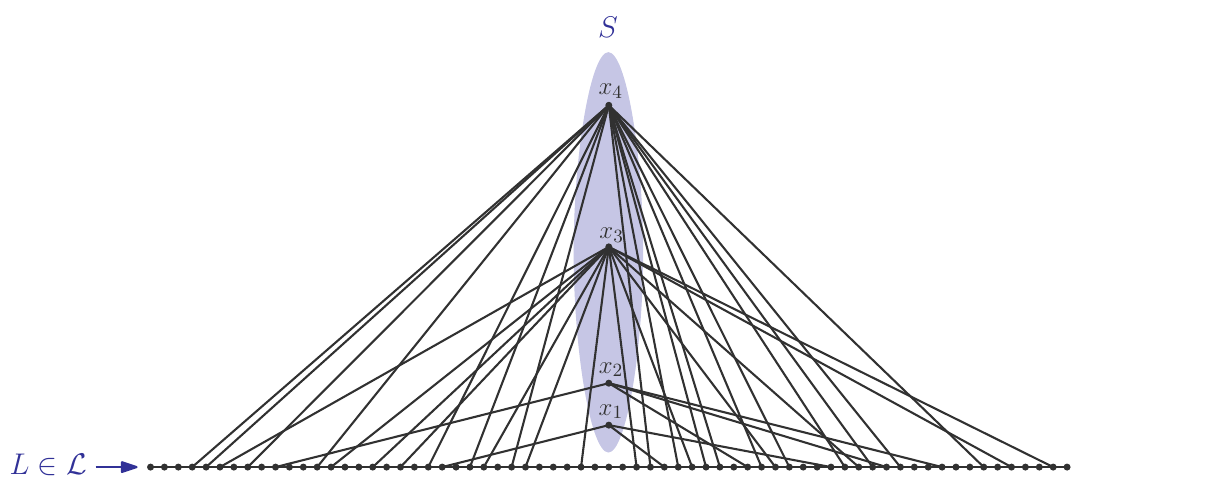}
    \caption{A $(4,1)$-constellation which is interrupted with $x_1\pre x_2\pre x_3\pre x_4$.}
    \label{fig:interrupted}
\end{figure}

For $q\in \poi$, we say a constellation $\mf{c}$ is \emph{$q$-zigzagged} if there is a linear order $\pre$ on $S_{\mf{c}}$ such that for all $x,y\in S_{\mf{c}}$ with $x\prec y$ and every $\mf{c}$-route $R$ from $x$ to $y$, there are fewer than $q$ vertices $z\in S_{\mf{c}}$ where $x\prec z\prec y$ and $z$ has no neighbor in $R$ (see Figure~\ref{fig:zigzag}).
\begin{figure}[t!]
    \centering
    \includegraphics[scale=0.6]{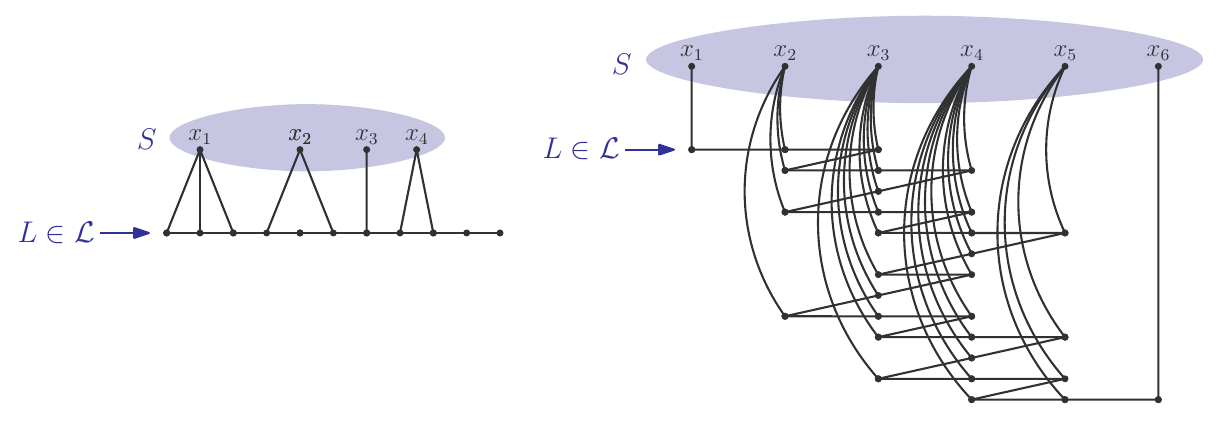}
    \caption{Left: A $(4,1)$-constellation which is $1$-zigzagged with $x_1\pre x_2\pre x_3\pre x_4$. Right: A $(6,1)$-constellation which is $1$-zigzagged with $x_1\pre x_2\pre x_3\pre x_4\pre x_5\pre x_6$.}
    \label{fig:zigzag}
\end{figure}
\medskip

We also need a canonical ``containment'' relation on constellations, which we define next. For constellations $\mf{b}$ and $\mf{c}$, we say \textit{$\mf{b}$ sits in $\mf{c}$} if
\begin{itemize}
    \item $S_{\mf{b}}\subseteq S_{\mf{c}}$;
    \item for every $L\in \mca{L}_{\mf{b}}$, there exists $M\in \mca{L}_{\mf{c}}$ such that $L\subseteq M$; and
    \item for every $M\in \mca{L}_{\mf{c}}$, there is at most one path $L\in \mca{L}_{\mf{b}}$ such that $L\subseteq M$.
\end{itemize}
See Figure~\ref{fig:sitting}. It follows that for constellations $\mf{b}$ and $\mf{c}$ where $\mf{b}$ sits in $\mf{c}$, the useful properties of $\mf{c}$ are also inherited by $\mf{b}$. In particular,
\begin{itemize}
\item $\mf{b}$ is an induced subgraph of $\mf{c}$;
    \item if $\mf{c}$ is $d$-ample for some $d\in \poi$, then so is $\mf{b}$;
    \item if $\mf{c}$ is interrupted, then so is $\mf{b}$; and
    \item if $\mf{c}$ is $q$-zigzagged for some $q\in \poi$, then so is $\mf{b}$.
\end{itemize}
\begin{figure}[t!]
    \centering
    \includegraphics[scale=0.6]{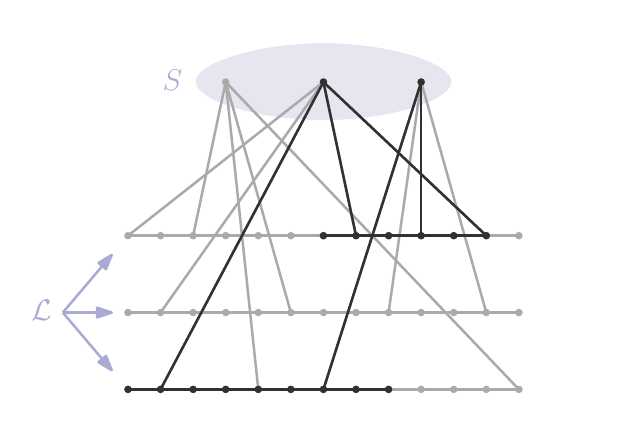}
    \caption{A $(2,2)$-constellation sitting in the $(3,3)$-constellation from Figure~\ref{fig:constellation}.}
    \label{fig:sitting}
\end{figure}
\medskip

We are now ready to state our second result:

\begin{restatable}{theorem}{mainconstellation}\label{thm:mainconstellation}
   For all $d,l,l',r,s,s',t\in \poi$, there are constants $f_{\ref{thm:mainconstellation}}=f_{\ref{thm:mainconstellation}}(d,l,l',r,s,s',t)\in \poi$ and $g_{\ref{thm:mainconstellation}}=g_{\ref{thm:mainconstellation}}(d,l,l',r,s,s',t)\in \poi$ such that for every $(f_{\ref{thm:mainconstellation}},g_{\ref{thm:mainconstellation}})$-constellation $\mf{c}$, one of the following holds.
   \begin{enumerate}[\rm (a)]
        \item\label{thm:mainconstellation_a}There is an induced subgraph of $\mf{c}$ isomorphic to either $K_{r,r}$ or a proper subdivision of $K_{2t+2}$.
        \item\label{thm:mainconstellation_b} There is a $d$-ample interrupted $(s,l)$-constellation which sits in $\mf{c}$. 
    \item\label{thm:mainconstellation_c} There is a $d$-ample $2t$-zigzagged $(s',l')$-constellation which sits in $\mf{c}$.
    \end{enumerate}
\end{restatable}

We also need the next result:

\begin{lemma}[Aboulker, Adler, Kim, Sintiari, Trotignon; see Lemma 3.6 in \cite{aboulker}] \label{lem:minor-to-ind}
    For every $r \in \poi$, there is a constant $f_{\ref{lem:minor-to-ind}} = f_{\ref{lem:minor-to-ind}}(r)$ such that if $G$ is a graph with an induced minor isomorphic to a subdivision of $W_{f_{\ref{lem:minor-to-ind}} \times f_{\ref{lem:minor-to-ind}}}$, then $G$ has an induced subgraph isomorphic to either a subdivision of $W_{r\times r}$ or the line graph of a subdivision of $W_{r\times r}$.
\end{lemma}

We can now state the main result of this paper, which follows directly from combining Theorems~\ref{thm:main_ind_minor} and \ref{thm:mainconstellation} with Lemma~\ref{lem:minor-to-ind}.

\begin{theorem}\label{thm:motherKtt}
For all $d,l,l',r,s,s'\in \poi$, there are constants $f_{\ref{thm:motherKtt}}=f_{\ref{thm:motherKtt}}(d,l,l',r,s,s')\in \poi$ and $g_{\ref{thm:motherKtt}}=g_{\ref{thm:motherKtt}}(d,l,l',r,s,s')\in \poi$ with the following property. Let $G$ be a graph which has an induced minor isomorphic to $K_{f_{\ref{thm:motherKtt}}, g_{\ref{thm:motherKtt}}}$. Then one of the following holds.
   \begin{enumerate}[\rm (a)]
        \item\label{thm:motherKtt_a}There is an induced subgraph of $G$ isomorphic to either $K_{r,r}$, a subdivision of $W_{r\times r}$, or the line graph of a subdivision of $W_{r\times r}$.
        \item\label{thm:motherKtt_b} There is a $d$-ample and interrupted $(s,l)$-constellation in $G$. 
    \item\label{thm:motherKtt_c} There is a $d$-ample and $2r^2$-zigzagged $(s',l')$-constellation in $G$.
    \end{enumerate}
\end{theorem}

It remains to prove Theorems~\ref{thm:main_ind_minor} and \ref{thm:mainconstellation}, which we will do in Sections~\ref{sec:finale1} and \ref{sec:bip}.
\medskip

This paper is organized as follows. In Section \ref{sec:zig}, we will give a construction showing that the third outcome of Theorem \ref{thm:motherKtt} cannot be strengthened to yield Pohoata-Davies-like graphs. In Section~\ref{sec:necessary}, we will prove that interrupted constellations and zigzagged constellations are both unavoidable as outcomes of Theorems~\ref{thm:mainconstellation} and \ref{thm:motherKtt}. Then, we proceed to the proof of Theorem \ref{thm:mainconstellation}: In Section \ref{sec:ample}, we show that we can restrict our attention to $d$-ample constellations. In Section \ref{sec:mixed}, we obtain a precursor to outcomes \ref{thm:mainconstellation_b} and \ref{thm:mainconstellation_c} of Theorem \ref{thm:mainconstellation} in which a ``mixture'' of the two cases holds. Section \ref{sec:zigzagged} recovers these two distinct outcomes, and Section \ref{sec:finale1} finishes the proof of Theorem \ref{thm:mainconstellation}. Finally, in Section \ref{sec:bip}, we prove Theorem \ref{thm:main_ind_minor}, hence completing the proof of Theorem~\ref{thm:motherKtt}.

\section{The zigzag graph} \label{sec:zig}

Before we continue to our main proofs, let us expand on Theorem \ref{thm:mainconstellation}\ref{thm:mainconstellation_c}. One might hope that this outcome could be refined further to coincide with the Pohoata-Davies graphs \cite{davies,pohoata2014unavoidable} or more precisely, with ``arrays,'' which are slight generalizations of Pohoata-Davies graphs \cite{tw12}. Let $\mf{c}$ be a constellation. We say that $\mf{c}$ is \emph{aligned} if there is a linear order $\pre$ on $S_\mf{c}$ such that every path $L$ in $\mathcal{L}_{\mf{c}}$ traverses neighbors of $S_{\mf{c}}$ in order, that is, there is a labelling $v_1, \dots, v_t$ of the vertices of $L$ such that for all $i, j \in \poi_t$: 
\begin{enumerate}
    \item $v_iv_j \in E(L)$ if and only if $|i-j| = 1$; and 
    \item if $s, s' \in S_\mf{c}$ with $s \prec s'$ such that $v_i \in N(s)$ and $v_j \in N(s')$, then $i < j$.
\end{enumerate}
Arrays are aligned constellations; Pohoata-Davies graphs have some further restrictions. 

Writing $x_1, \dots, x_s$ for the vertices of $S_{\mf{c}}$, each path $L$ of $\mathcal{L}_{\mf{c}}$ gives rise to a sequence $A_L$ (unique up to reversal) obtained by recording, as we traverse $L$, the indices of vertices in $S_{\mf{c}}$ whose neighbors we encounter. For example, in Figure \ref{fig:zigzag}, this sequence would read $1, 1, 1, 2, 2, 3, 4, 4$ for the graph of the left, and begin with $1, 2, 3, 2, 3, 4$ for the graph on the right. A constellation is aligned (with the linear order $x_1 \pre \cdots \pre x_s$), then, if for every $L \in \mathcal{L}_{\mf{c}}$, we have $A_L = 1, 2, \dots, s$ after omitting consecutive occurrences of the same number (and up to reversal). 

Let us say that a sequence $A = a_1, \dots, a_t$ of integers is \emph{smooth} if consecutive entries differ by at most 1. Thus, a constellation $\mf{c}$ is $1$-zigzagged if and only if there is a linear order on $\mf{c}$ such that $A_L$ is smooth for every $L \in \mathcal{L}_{\mf{c}}$. 

One might hope that if $\mf{c}$ is a sufficiently large 1-zigzagged constellation, then there is a large aligned constellation which sits in $\mf{c}$. Translating this to sequences, we would hope that for every smooth sequence $A = a_1, \dots, a_t$ with $|\{a_1, \dots, a_t\}|$ sufficiently large, there is a large subset $B \subseteq \{a_1, \dots, a_t\}$ as well as $i, j \in \poi_t$ such that the sequence $A^{B, i, j}$, obtained from $a_i, \dots, a_j$ by skipping all entries which are not in $B$ and then omitting consecutive entries that are equal, is a sequence containing each element of $B$ exactly once. Let us say that $B$ is an \emph{alignment in $A$} in this case; we call $|B|$ the \emph{size} of the alignment. We will show that for every $n$, there is a smooth sequence containing the numbers $1, \dots, n$ with no alignment of size 4. We call this sequence the \textit{zigzag sequence}. The associated constellation is 1-zigzagged, but no aligned constellation with a stable set of size 4 sits in it. 

To define the zigzag sequence, we need some further definitions. Given a sequence $A = a_1, \dots, a_t$, we define $A + 1$ to be the sequence $a_1+1, a_2+1, \dots, a_t+1$. Moreover, we write $A^{-1}$ for the \emph{reverse of $A$}, that is, the sequence $a_t, a_{t-1}, \dots, a_1$. We define $A^*$ to be the sequence $a_1, \dots, a_{t-1}$. 

We are now ready to define the \emph{zigzag sequence} $Z_n$ for $n \in \mathbb{N}_{\geq 2}$. We let $Z_2 = 1,2$ and $Z_3 = 1, 2,3$. For $n \geq 4$, we define to be the sequence $Z_n = Z_{n-1}^*, ((Z_{n-2}+1)^{-1})^*, (Z_{n-1})+1$.  
Thus, $$Z_4 = 1, 2, 3, 2, 3, 4$$
and 
$$Z_5 = 1, 2, 3, 2, 3, 4, 3, 2, 3, 4, 3, 4, 5.$$
From the definition, it follows that $Z_n$ is a smooth sequence containing the numbers $1, \dots, n$. In addition, replacing each entry $a_i$ of $Z_n$ with $n+1-a_i$ yields the sequence $(Z_n)^{-1}$. We write $l_n$ for the length of $Z_n$. 

\begin{lemma} \label{lem:align}
    $Z_n$ contains no 4-alignment. 
\end{lemma}
\begin{proof}
    Suppose for a contradiction that $\{p, q, r, s\}$ is an alignment in the sequence $A = z_i, \dots, z_j$ for some $i < j$, where $Z_n = z_1, \dots, z_{l_n}$. We choose this counterexample with $n$ minimum. We may further choose this sequence with $j-i$ minimum, and so we may assume that $z_i = p$ and $z_j = s$ and $z_k \not\in \{p, s\}$ for $k \in \{i+1, \dots, j-1\}$. By symmetry, we may assume that $p < s$. It follows that $z_{i+1}, \dots, z_{j-1} \in \{p+1, \dots, s-1\}$.  

    We claim that $p = 1$. Suppose not, so $p \geq 2$. It follows that all entries of $A$ except possibly $z_i$ are strictly larger than 2, and so $A$ is contained in either $(Z_{n-1})+1$ or $Z_{n-1}^*, ((Z_{n-2}+1)^{-1})^*, 2 = Z_{n-1}^*, ((Z_{n-2}+1)^{-1})$. The former contradicts the minimality of $n$, and so the latter case holds. In particular, $A$ does not contain $n$, and so $s \leq n-1$. Since the first entry of $((Z_{n-2}+1)^{-1})$ is $n-1$, it follows that $A$ is contained in one of $Z_{n-1}^*, n-1 = Z_{n-1}$ or $((Z_{n-2}+1)^{-1})$; however, both contradict the minimality of $n$. Thus, $p = 1$, and by symmetry, $s = n$. It follows that $i = 1$ and $j = l_n$, and so $A = Z_n$. 

    By symmetry, we may assume that $q < r$. Then, 
    $$A = Z_n = Z_{n-1}^*, ((Z_{n-2}+1)^{-1})^*, (Z_{n-1})+1$$
    and $Z_{n-1}^*$ contains $q$ (and possibly $r$), $((Z_{n-2}+1)^{-1})^*$ contains $r$ (and possibly $q$), and $(Z_{n-1})+1$ contains $q$ (and possibly $r$); but this contradicts the assumption that $\{p, q, r, s\}$ is an alignment in $A$, hence completing the proof.    
\end{proof}

The \emph{zigzag graph} $GZ_n$ is defined as follows. Let $Z_n = z_1, \dots, z_{l_n}$. We let $V(GZ_n) = \{p_1, \dots, p_{l_n}, x_1, \dots, x_n\}$ and 
$$E(GZ_n) = \{p_ip_{i+1} : i \in \{1, \dots, l_n-1\} \} \cup \{x_ip_j : i \in \{1, \dots, n\}, j \in \{1, \dots, l_n\}, z_j = i\}.$$
This is an $(n, 1)$-constellation which is $1$-zigzagged; Figure \ref{fig:zigzag} (right) shows $GZ_6$. By making $l$ copies of the path $p_1, \dots, p_{l_n}$, we obtain an $(n, l)$-constellation $GZ_n^l$. Using Lemma \ref{lem:align}, one can show that $GZ_n^l$ contains no large array (we omit the proof). It follows that it is not possible to simplify  Theorem \ref{thm:mainconstellation}\ref{thm:mainconstellation_c} to the case of arrays. We leave open the question of whether ``$2t$-zigzagged'' can be improved to ``$1$-zigzagged.''

\section{The interrupted and the zigzagged outcome are both necessary}\label{sec:necessary}
Our goal in this section is to prove that Theorems~\ref{thm:mainconstellation} and \ref{thm:motherKtt} are ``best possible'' in the sense that both the interrupted and the zigzagged constellations are unavoidable outcomes of those theorems.

\subsection{The interrupted case} For $r\in \poi$, we denote by $T_r$ the full binary tree of radius $r$ (on $2^{r+1}-1$ vertices). It is well-known \cite{GM1} that for every $r\in \poi$, all subdivisions of $T_{2r}$ and their line graphs have \textit{pathwidth} at least $r$ (where the pathwidth of a graph $G$ is denoted by $\pw(G)$; see \cite{diestel} for a definition). 

Note that all constellations are $K_4$-free, and all ample constellations are $K_{3,3}$-free. Thus, to show that interrupted constellations cannot be omitted from the outcomes of Theorems~\ref{thm:mainconstellation} and \ref{thm:motherKtt}, it is enough to prove the following (we remark that the second bullet is not necessary for our purposes here; however, it will be used in a future paper \cite{tw18}):

\begin{theorem}\label{thm:interruptedoutcome}
Let $\mf{c}$ be an ample interrupted constellation. Then $\mf{c}$ has no induced subgraph isomorphic to any of the following.
    \begin{itemize}
       \item An ample $q$-zigzagged $\left(3q+6,6\binom{q+2}{3}\right)$-constellation, for $q\in \poi$.
        \item A subdivision of $T_{7}$ or the line graph of a subdivision of $T_{7}$. 
        \item A subdivision of $W_{6\times 6}$ or the line graph of a subdivision of $W_{6\times 6}$. 
    \end{itemize}
\end{theorem}

The proof of Theorem~\ref{thm:interruptedoutcome} relies on the next two lemmas:

\begin{lemma}\label{lem:deathstar}
  Let  $\mf{c}$ be an ample interrupted constellation. Then for every two anticomplete induced subgraphs $X_1,X_2$ of $\mf{c}$, there exists $i\in \{1,2\}$ such that each component of $X_i$ intersects $S_{\mf{c}}$ in at most one vertex.
\end{lemma}
\begin{proof}
Suppose for a contradiction that there are two anticomplete subsets $X_1,X_2$ of $V(\mf{c})$ such that for every $i\in \{1,2\}$, there is a component $K_i$ of $G[X_i]$ with $|K_i\cap S_{\mf{c}}|\geq 2$; in particular, $X_i\cap S_{\mf{c}}\neq \varnothing$. Let $x_1\prec\cdots\prec x_{s}$ be the linear order on the vertices in $S_{\mf{c}}$ with respect to which $\mf{c}$ is interrupted. Let $j\in \poi_{s}$ be maximum such that $x_j\in X_1\cup X_2$. Without loss of generality, we may assume that $x_j\in X_1$ (and so $x_j\notin X_2$). Let $R$ be a shortest path in $K_2$ whose ends $x_{i},x_{i'}$ belong to $K_2\cap S_{\mf{c}}$ (note that $R$ exists because $K_2$ is connected and $|K_2\cap S_{\mf{c}}|\geq 2$). By the choice of $R$, we have $R^*\subseteq K_2\setminus S_{\mf{c}}\subseteq L$ and by the choice of $j$, we have $i,i'<j$. But now  $x_j\in X_1$ has a neighbor in $R^*\subseteq X_2$ (because $\mf{c}$ is interrupted), a contradiction to the assumption that $X_1$ and $X_2$ are anticomplete in $\mf{c}$.
\end{proof}

\begin{lemma}\label{lem:eez_zigzag2}
Let $c,q\in \poi$ and let  $\mf{c}$ be an ample $\left(2c+3q, 2c{c+q-1\choose c}\right)$-constellation that is $q$-zigzagged. Then there are anticomplete subsets $X,Y$ of $V(\mf{c})$ with $\tw(X),\tw(Y)\geq c$.
\end{lemma}
\begin{proof}
    Since $|S_{\mf{c}}|=2c+3q$, we may choose $x_1,x_2\in S_{\mf{c}}$ and pairwise disjoint subsets $Q,S_1,S_2\subseteq S_{\mf{c}}$ such that $|Q|=q$, $|S_1|=|S_2|=c+q-1$ and
    $x_1\prec S_1\prec Q\prec S_2\prec x_2$.

For each $i\in \{1,2\}$ and every $L\in \mca{L}_{\mf{c}}$, let $R_{i,L}$ be a $\mf{c}$-route from $x_i$ to a vertex in $Q$ with $|R_{i,L}|$ as small as possible.  We claim that:

\sta{\label{st:eachpathanti} For every $L\in \mca{L}_{\mf{c}}$, the sets $S_1$ and $R^*_{2,L}$ are anticomplete in $\mf{c}$, and the sets $S_2$ and $R^*_{1,L}$ are anticomplete in $\mf{c}$.}

Suppose not. Then, by symmetry, we may assume that some for some $L\in \mca{L}_{\mf{c}}$, there is a vertex $u\in S_2$ with a neighbor in $R^*_{1,L}$. Let $y\in Q$ be the end of $R_{1,L}$ other than $x$. Since $\mf{c}$ is ample, it follows that there is a $\mf{c}$-route $R'$ from $x_1$ to $u$ with $R'^*\subseteq R_{1,L}^*\setminus N_R(y)$. Since $\mf{c}$ is $q$-zigzagged, and since $x_1\prec Q\prec u$, it follows that some vertex $z\in Q$ has a neighbor in $R'^*$. Consequently, there is a $\mf{c}$-route $R''$ from $x$ to $z\in Q_1$ with $R''^*\subseteq R'^*\subseteq R_{1,L}^*\setminus N_R(y)$, and so $|R''|<|R_{1,L}|$. This violates the choice of $R_{1,L}$, hence proving \eqref{st:eachpathanti}.
\medskip

Let $\mca{L}_1,\mca{L}_2\subseteq \mca{L}_{\mf{c}}$ be disjoint with $|\mca{L}_1|=|\mca{L}_2|=c{c+q-1\choose q}$.  Since $\mf{c}$ is $q$-zigzagged and since $|S_1|=|S_2|=c+q-1$, it follows that for every $i\in \{1,2\}$ and every $L\in \mca{L}_{i}$, there is a $c$-subset $S'_{i,L}$ of $S_i$ such that every vertex in $S'_{i,L}$ has a neighbor in $R^*_{i,L}$. Since $|\mca{L}_1|=|\mca{L}_2|=c{c+q-1\choose q}$, it follows that for every $i\in \{1,2\}$, there is a $c$-subset $S'_i$ of $S_i$ and a $c$-subset $\mca{L}'_i$ of $\mca{L}_i$ such that for every $L\in \mca{L}'_i$, we have $S'_{i,L}=S'_i$.

Now, let
$$X=(S'_1,\{R^*_{1,L}:L\in \mca{L}'_1\})$$
and let
$$Y=(S'_2,\{R^*_{2,L}:L\in \mca{L}'_2\}).$$

By \eqref{st:eachpathanti} and since $\mf{c}$ is a constellation, it follows that $X$ and $Y$ are anticomplete $(c,c)$-constellations in $\mf{c}$. In particular, $X$ and $Y$ are anticomplete induced subgraphs of $\mf{c}$ each with an induced minor isomorphic to $K_{c,c}$. Hence, we have $\tw(X),\tw(Y)\geq c$. This completes the proof of Lemma~\ref{lem:eez_zigzag2}.
\end{proof}

\begin{proof}[Proof of Theorem~\ref{thm:interruptedoutcome}]
By Lemma~\ref{lem:deathstar}, $\mf{c}$ has no two anticomplete induced subgraphs, each of pathwidth more than $2$. On the other hand, by Lemma~\ref{lem:eez_zigzag2}, for every $q\in \poi$, every ample $q$-zigzagged $\left(3q+6,6\binom{q+2}{3}\right)$-constellation has two anticomplete induced subgraphs, each of treewidth (and so pathwidth) at least $3$. Therefore, $\mf{c}$ has no induced subgraph isomorphic to an ample $q$-zigzagged $\left(3q+6,6\binom{q+2}{3}\right)$-constellation, where $q\in \poi$.

Moreover, observe that (the line graph of) every subdivision of $T_{7}$ has two anticomplete induced subgraphs each isomorphic to (the line graph of) a subdivision of $T_{6}$, and so each of pathwidth at least $3$. Thus, $\mf{c}$ has no induced subgraph isomorphic to a subdivision of $T_{7}$ or the line graph of a subdivision of $T_{7}$. Finally, observe that (the line graph of) every subdivision of $W_{6\times 6}$ has two anticomplete induced subgraphs each isomorphic to (the line graph of) a subdivision of $W_{3\times 3}$, and so each of treewidth (and so pathwidth) at least $3$. Hence, $\mf{c}$ has no induced subgraph isomorphic to a subdivision of $W_{6\times 6}$  or the line graph of a subdivision of $W_{6\times 6}$. This completes the proof of Theorem~\ref{thm:interruptedoutcome}.
\end{proof}

\subsection{The zigzagged case} Again, since constellations are $K_4$-free and ample constellations are $K_{3,3}$-free, the following suffices to show that zigzagged constellations cannot be omitted from the outcomes of Theorems~\ref{thm:mainconstellation} and \ref{thm:motherKtt}:

\begin{theorem}\label{thm:zigzaggedoutcome}
Let $q\in \poi$ and let $\mf{c}$ be an ample $q$-zigzagged constellation. Then $\mf{c}$ has no induced subgraph isomorphic to any of the following.
    \begin{itemize}
        \item An ample interrupted $(2q+6)$-constellation.
        \item A subdivision of $T_{64q^2}$ or the line graph of a subdivision of $T_{64q^2}$.
        \item A subdivision of $W_{2^{64q^2}\times 2^{64q^2}}$ or the line graph of a subdivision of $W_{2^{64q^2}\times 2^{64q^2}}$.
    \end{itemize}
\end{theorem}

The proof is in several steps. First, we show that:

\begin{lemma}\label{lem:zigzagged}
  Let $q\in \poi$ and let $\mf{c}$ be an ample $q$-zigzagged constellation. Then there is no induced subgraph $H$ of $\mf{c}$ isomorphic to a proper subdivision of $K_{1,2q+1}$ where the center and the leaves of $H$ have degree at least $4$ in $\mf{c}$.
\end{lemma}

\begin{proof}
Suppose not. Then there are paths $P_1,\ldots, P_{2q+1}$ of non-zero length in $\mf{c}$, all sharing an end $x$, such that $(P_i\setminus \{x\}:i\in \poi_{2q+1})$ are pairwise anticomplete in $\mf{c}$, and for each $i\in \poi_{2q+1}$, the ends of $P_i$ have degree at least $4$ in $\mf{c}$. Choose $P_1,\ldots, P_{2q+1}$ with $P_1\cup \cdots\cup P_{2q+1}$ minimal subject to the above property. For every $i\in \poi_{2q+1}$, let $x_i$ be the end of $P_i$ different than $x$; thus, both $x$ and $x_i$ have degree at least $4$ in $\mf{c}$.  Since $\mf{c}$ is ample, it follows that all vertices of degree at least $4$ in $\mf{c}$ belong to $S_{\mf{c}}$. In particular, we have $x,x_1,\ldots, x_{2q+1}\in S_{\mf{c}}$. Moreover, from the minimality of $P_1\cup \cdots\cup P_{2q+2}$, it follows that for every $i\in \poi_{2q+2}$, we have $P_i^*\cap S_{\mf{c}}=\varnothing$, and so $P_i$ is $\mf{c}$-route from $x$ to $x_i$. Let $\prec$ be the linear order on $S_{\mf{c}}$ with respect to which $\mf{c}$ is $q$-zigzagged. Since $x,x_1,\ldots, x_{2q+1}\in S_{\mf{c}}$, it follows that there are $i_1,\ldots, i_{q+1}\in \poi_{2q+1}$ such that either $x\prec x_{i_1}\prec \cdots \prec x_{i_{q+1}}$ or $x\succ x_{i_1}\succ \cdots \succ x_{i_{q+1}}$. But now since $\mf{c}$ is $q$-zigzagged, it follows that $x_{i_j}$ has a neighbor in $P_{i_{q+1}}$ for some $j\in \poi_{q}$, a contradiction to the assumption that $(P_i\setminus \{x\}:i\in \poi_{2q+1})$ are pairwise anticomplete in $\mf{c}$.
\end{proof}

We need the following result from \cite{approxpw}:

\begin{theorem}[Groenland, Joret, Nadra, Walczak  \cite{approxpw}]\label{thm:pwsub}
   Let $r,t\in \poi$ and let $G$ be a graph of treewidth at most $t-1$. Then either $G$ has pathwidth at most $rt+1$ or $G$ has a subgraph isomorphic to a subdivision of $T_{r+1}$.
\end{theorem}

Given a graph $H$, by a \textit{leaf-extension} of $H$ we mean a graph $G$ obtained from $H$ by adding a set $S$ of pairwise non-adjacent vertices, each with exactly one neighbor in $V(H)$.

\begin{lemma}\label{lem:pwsub}
    Let $w\in \poi$, let $H$ be a graph of pathwidth at most $w$ and let $G$ be a subdivision of a leaf-extension of $H$. Then $\pw(G)\leq 2w^2+3w+2$.
\end{lemma}
\begin{proof}
    Since $\tw(H)\leq \pw(H)\leq w$ and since $G$ is a subdivision of a leaf-extension of $H$, it follows that $\tw(G)\leq w$. Moreover, since $\pw(H)\leq w$, it follows that $H$ has no subgraph isomorphic to a subdivision of $T_{2w}$, and so $G$ has no subgraph isomorphic to a subdivision of $T_{2w+2}$. But now by Theorem~\ref{thm:pwsub} applied to $G$, we have $\pw(G)\leq (w+1)(2w+1)+1=2w^2+3w+2$.
\end{proof}

We continue with more definitions. Let $G$ be a graph. For $X\subseteq V(G)$, the \textit{$X$-contraction of $G$} is the graph obtained from $G$ by contracting all edges in $G$ with at least one end in $X$ (and removing all loops and the parallel edges arising in this process). It follows that the $X$-contraction of $G$ is an induced minor of $G$. By a \textit{star-contraction of $G$}, we mean the $X$-contraction of $G$ for some $X\subseteq V(G)$ such that the vertices in $X$ are pairwise at distance at least $3$ in $G$ (in particular, $X$ is a stable set in $G$).

Let $\mf{c}$ be a constellation and let $H$ be an induced subgraph of $\mf{c}$. We define $\mf{c}_H$ to be the graph with vertex set $V(H)\cap S_{\mf{c}}$ such that for all distinct $x,x'\in V(H)\cap S_{\mf{c}}$, we have $xx'\in E(\mf{c}_H)$ if and only if there is a $\mf{c}$-route $R$ from $x$ to $x'$ where $R\subseteq V(H)$ and $(V(H)\cap S_{\mf{c}})\setminus \{x,x'\}$ is anticomplete to $R$ in $H$.

We will use the following observation in the proof of Theorem~\ref{thm:zigzaggedoutcome}. Note that \ref{obs:zigzag}\ref{obs:zigzag_a} is immediate from the assumption that $\mf{c}$ is ample, and  \ref{obs:zigzag}\ref{obs:zigzag_b} follows easily from the assumption that $H$ has no cycle on four or more vertices; we omit the details.

\begin{observation}\label{obs:zigzag}
Let  $\mf{c}$ be an ample constellation and let $H$ be an induced subgraph of  $\mf{c}$ with no cycle on four or more vertices. Let $J$ be the $(V(H)\cap S_{\mf{c}})$-contraction of $H$. Then the following hold.
\begin{enumerate}[{\rm (a)}]
    \item\label{obs:zigzag_a} $J$ is a star contraction of $H$.
    
      \item\label{obs:zigzag_b} $J$ is isomorphic to a subdivision of a leaf-extension of $\mf{c}_H$.
\end{enumerate}
\end{observation}

We will also use the next observation (again, we omit the proof as it is easy).

\begin{observation}\label{obs:trees}
Let $r\in \poi$. Then there is an induced subgraph of (the line graph of) $W_{2^r\times 2^r}$ isomorphic to (the line graph of) a proper subdivision of $T_{r}$.
\end{observation}

Let us now prove Theorem~\ref{thm:zigzaggedoutcome}:

\begin{proof}[Proof of Theorem~\ref{thm:zigzaggedoutcome}]
First, assume that $\mf{c}$ has an induced subgraph $\mf{b}$ which is a $1$-ample interrupted $(2q+6)$-constellation. Let $x_1\prec \cdots\prec x_{2q+6}$ be the linear order on $S_{\mf{b}}$ with respect to which $\mf{b}$ is interrupted. Since $\mf{b}$ is interrupted, it follows that every vertex in $S_{\mf{b}}\setminus \{x_1, x_2, x_3,x_4\}$ has degree at least $4$ in $\mf{b}$ (and so in $\mf{c}$), and there is a $\mf{b}$-route $R_i$ from $x_{2q+6}$ to $x_i$, for each $i\in \poi_{2q+5}$, such that $(R_i\setminus \{x_{2q+6}\}:i\in \poi_{2q+5})$ are pairwise anticomplete in $\mf{b}$ (hence in $\mf{c})$. But now $H=\bigcup_{i=5}^{2q+5}P_i$ is an induced subgraph of $\mf{c}$ isomorphic to a proper subdivision of $K_{1,2q+1}$ where the center and the leaves of $H$ of degree at least $4$ in $\mf{c}$, a contradiction to Lemma~\ref{lem:zigzagged}. Therefore, $\mf{c}$ has no induced subgraph $\mf{b}$ which is a $1$-ample interrupted $(2q+6)$-constellation.

Next, we prove that $\mf{c}$ has no induced subgraph isomorphic to a subdivision of $T_{64q^2}$ or the line graph of a subdivision of $T_{64q^2}$. This, combined with Observations~\ref{obs:trees}, will also imply that $\mf{c}$ has no induced subgraph isomorphic to a subdivision of $W_{2^{64q^2}\times 2^{64q^2}}$ or the line graph of a subdivision of $W_{2^{64q^2}\times 2^{64q^2}}$.

Suppose for a contradiction that $\mf{c}$ has an induced subgraph isomorphic to (the line graph of) a subdivision of $T_{64q^2}$. Then there is a $(\geq 3)$-subdivision $T$ of $T_{16q^2}$ such that $\mf{c}$ has an induced subgraph $H$ isomorphic to (the line graph of) $T$. Let $J$ be the $(V(H)\cap S_{\mf{c}})$-contraction of $H$. By Observation~\ref{obs:zigzag}\ref{obs:zigzag_a}, $J$ is a star contraction of $H$. From this and the fact that $T$ is a $(\geq 3)$-subdivision of $T_{16q^2}$, it follows that $J$ has a minor isomorphic to $T_{16q^2}$. In particular, we have $\pw(J)\geq \pw(T_{16q^2})\geq 8q^2$.

Let us now show that:

\sta{\label{st:transition} There is an enumeration $x_1,\ldots, x_s$ of the vertices of $\mf{c}_H$ such that if $x_ix_j\in E(\mf{c}_H)$ for some $i,j\in \poi_s$ with $i<j$, then we have $j-i\leq q$.}

Let $\prec$ be the linear order on $S_{\mf{c}}$ with respect to which $\mf{c}$ is $q$-zigzagged, and let $V(\mf{c}_H)=V(H)\cap S_{\mf{c}}=\{x_1,\ldots, x_s\}$ such that  $x_1\prec \cdots \prec x_s$. Suppose that $x_ix_j\in E(\mf{c}_H)$ for some $i,j\in \poi_s$ with $j>i+q$. Then, by the definition of $\mf{c}_H$, there is a $\mf{c}$-route $R$ with $R\subseteq V(H)$ such that $\{x_{i+1},\ldots, x_{i+q}\}\subseteq (V(H)\cap S_{\mf{c}})\setminus \{x_i,x_j\}$ is anticomplete to $R$, a contradiction to the assumption that $\mf{c}$ is zigzagged. This proves \eqref{st:transition}.
\medskip

From \eqref{st:transition}, we deduce that $\pw(\mf{c}_H)\leq q$. Moreover, by Observation~\ref{obs:zigzag}\ref{obs:zigzag_b}, $J$ is isomorphic to a subdivision of a leaf-extension of $\mf{c}_H$. But now by Lemma~\ref{lem:pwsub}, we have $\pw(J)\leq 2q^2+3q+2<8q^2$, a contradiction. This completes the proof of Theorem~\ref{thm:zigzaggedoutcome}.
\end{proof}

\section{Amplification} \label{sec:ample}
The next four sections will be devoted to the proof of Theorem~\ref{thm:mainconstellation}. The first step is to pass to a sufficiently ample constellation; doing so makes it easier to find routes which are pairwise anticomplete. This will be done using the following lemma:

\begin{lemma}\label{lem:getample}
For all $d,l,r,r',s\in \poi$, there are constants $f_{\ref{lem:getample}}=f_{\ref{lem:getample}}(d,l,r,r',s)\in \poi$ and $g_{\ref{lem:getample}}=g_{\ref{lem:getample}}(d,l,r,r',s)\in \poi$ such that for every $(f_{\ref{lem:getample}},g_{\ref{lem:getample}})$-constellation $\mf{c}$, one of the following holds.
 \begin{enumerate}[\rm (a)]
    \item\label{lem:getample_a} There is an induced subgraph of $\mf{c}$ which is isomorphic to either $K_{r,r}$ or a proper subdivision of $K_{r'}$.
    \item\label{lem:getample_b} There is a $d$-ample $(s,l)$-constellation $\mf{d}$ with $S_{\mf{d}} \subseteq S_{\mf{c}}$ and $\mca{L}_{\mf{d}} \subseteq \mca{L}_{\mf{c}}$. 
    \end{enumerate}
\end{lemma}

We will deduce Lemma~\ref{lem:getample} from a straightforward application of the next two results:

\begin{lemma}[Alecu, Chudnovsky, Hajebi and   Spirkl; see Lemma 6.2 in  \cite{tw9}]\label{lem:getampletw9}
For all $l,m,s\in \poi$, there is a constant $f_{\ref{lem:getampletw9}}=f_{\ref{lem:getampletw9}}(l,m,s)\in \poi$ such that for every $(f_{\ref{lem:getampletw9}},l+m^2-1)$-constellation $\mf{c}$ and every $d\in \poi$, one of the following holds.
 \begin{enumerate}[\rm (a)]
        \item\label{lem:getampletw9_a}There is a subgraph of $\mf{c}$ which is isomorphic to a $(\leq  d)$-subdivision of $K_m$.
    \item\label{lem:getampletw9_b} There is a $d$-ample $(s,l)$-constellation $\mf{d}$ with $S_{\mf{d}} \subseteq S_{\mf{c}}$ and $\mca{L}_{\mf{d}} \subseteq \mca{L}_{\mf{c}}$. 
    \end{enumerate}
\end{lemma}

\begin{lemma}[Hajebi \cite{pinned}; see also Dvo\v{r}\'ak \cite{dvorak}]\label{lem:pinnedasconj}
   For all $d,r,r'\in \poi$, there is a constant $f_{\ref{lem:pinnedasconj}}=f_{\ref{lem:pinnedasconj}}(d,r,r')\in \poi$ with the following property. Let $G$ be a $K_{r}$-free and $K_{r,r}$-free graph which has a subgraph isomorphic to a $(<d)$-subdivision of $K_{f_{\ref{lem:pinnedasconj}}}$. Then $G$ has an induced subgraph isomorphic to a proper $(<d)$-subdivision of $K_{r'}$.
\end{lemma}

\begin{proof}[Proof of Lemma~\ref{lem:getample}]
    Let $m=f_{\ref{lem:pinnedasconj}}(d,\max\{r,4\},r')$, let
    $$f_{\ref{lem:getample}}=f_{\ref{lem:getample}}(d,l,r,r',s)=f_{\ref{lem:getampletw9}}(l,m,s)$$
    and let
$$g_{\ref{lem:getample}}=g_{\ref{lem:getample}}(d,l,r,r',s)=l+m^2-1.$$
Let $\mf{c}$ be a $(f_{\ref{lem:getample}},g_{\ref{lem:getample}})$-constellation. Apply Lemma~\ref{lem:getampletw9} to $\mf{c}$. Since \ref{lem:getample}\ref{lem:getample_b} and \ref{lem:getampletw9}\ref{lem:getampletw9_b} are identical, we may assume that \ref{lem:getampletw9}\ref{lem:getampletw9_a} holds, that is, $\mf{c}$ has a subgraph isomorphic to a $(\leq  d)$-subdivision of $K_m$. Now, observe that $\mf{c}$ is $K_4$-free. Moreover, if $G$ has a subgraph isomorphic to $K_{r,r}$, then \ref{lem:getample}\ref{lem:getample_a} holds. So we may assume that $\mf{c}$ is $K_{r,r}$-free, as well. This, along with the choice of $m$, allows for an application of Lemma~\ref{lem:pinnedasconj}. We conclude that $\mf{c}$ has an induced subgraph isomorphic to a proper subdivision of $K_{r'}$, which in turn implies that \ref{lem:getample}\ref{lem:getample_a} holds, as desired.
\end{proof}

\section{Mixed constellations} \label{sec:mixed}
In this section, we prove a preliminary version of Theorem~\ref{thm:mainconstellation} in which the interrupted and the zigzagged outcomes are replaced by a ``mixture'' of the two.

Recall that for $q\in \poi$, a constellation $\mf{c}$ is $q$-zigzagged if there is a linear order $\pre$ on $S_{\mf{c}}$  such for all $x,y\in S_{\mf{c}}$ with $x\prec y$, every $\mf{c}$-route $R$ from $x$ to $y$ and every $q$-subset $Q$ of $S_{\mf{c}}$ with $x\prec Q\prec y$, some vertex in $Q$ has a neighbor in $R$. In a ``mixed'' constellation, we relax this condition by allowing $Q$ to also contain a controlled number of vertices in $S_{\mf{c}}$ that are \textbf{not} between the ends of $R$.

More precisely, for $q\in \poi$ and $q^-,q^+\in \poi\cup \{0\}$, we say a constellation $\mf{c}$ is \textit{$(q^-,q,q^+)$-mixed} if there is a linear order $\pre$ on $S_{\mf{c}}$ for which the following holds. Let $x,y\in S_{\mf{c}}$ with $x\prec y$, let $R$ be a $\mf{c}$-route from $x$ to $y$ and let $Q^-, Q,Q^+\subseteq S_{\mf{c}}$ such that $|Q^-|=q^-$, $|Q|=q$, $|Q^+|=q^+$ and $Q^-\prec x\prec Q \prec y\prec Q^+$.
Then some vertex in $Q^-\cup Q\cup Q^+$ has a neighbor in $R$. In particular, observe that $\mf{c}$ is $(0,q,0)$-mixed if and only if $\mf{c}$ is $q$-zigzagged. Also, for constellations $\mf{b}$ and $\mf{c}$ where $\mf{b}$ sits in $\mf{c}$, if $\mf{c}$ is $(q^-,q,q^+)$-mixed for some $q^-,q^+\in \poi\cup \{0\}$ and $q\in \poi$, then so is $\mf{b}$.
\medskip

We show that:

\begin{lemma}\label{lem:gettingmixed}
   For all $d,l,r,s,t\in \poi$, there are constants $f_{\ref{lem:gettingmixed}}=f_{\ref{lem:gettingmixed}}(d,l,r,s,t)\in \poi$ and $g_{\ref{lem:gettingmixed}}=g_{\ref{lem:gettingmixed}}(d,l,r,s,t)\in \poi$ with the following property. Let $\mf{c}$ be a $(f_{\ref{lem:gettingmixed}},g_{\ref{lem:gettingmixed}})$-constellation. Then one of the following holds.
   \begin{enumerate}[\rm (a)]
        \item\label{lem:gettingmixed_a} There is an induced subgraph  of $\mf{c}$ that is isomorphic to either $K_{r,r}$ or a proper subdivision of $K_{2t+2}$.
    \item\label{lem:gettingmixed_b} There is a $d$-ample and $(2t,2t,2t)$-mixed $(s,l)$-constellation $\mf{b}$ which sits in $\mf{c}$.
    \end{enumerate}
\end{lemma}
The proof uses the following version of Ramsey's theorem:

\begin{theorem}[Ramsey \cite{multiramsey}]\label{thm:multiramsey}
For all $l,m,n\in \poi$, there is a constant $f_{\ref{thm:multiramsey}}=f_{\ref{thm:multiramsey}}(l,m,n)\in \poi$ with the following property. Let $U$ be a set of cardinality at least $f_{\ref{thm:multiramsey}}$ and let $F$ be a non-empty set of cardinality at most $l$. Let $\Phi:\binom{U}{m}\rightarrow F$ be a map. Then there exist $i\in F$ and an $n$-subset $Z$ of $U$ such that $\Phi(X)=i$ for all $X\in \binom{Z}{m}$.
\end{theorem}

The main idea of the proof is the following: If there is a large $X$ subset of $S_{\mf{c}}$ such that for every two vertices in $X$, there are many $\mf{c}$-routes between them avoiding neighbors of the remaining vertices in $Q$, then this is an induced subdivision of $K_{|Q|}$. So we apply Theorem \ref{thm:multiramsey}, asking for each subset $X$ which $\mf{c}$-routes exist and avoid neighbors of other vertices. 

\begin{proof}[Proof of Lemma~\ref{lem:gettingmixed}]
Let $$m=f_{\ref{thm:multiramsey}}\left(\binom{2t+2}{2}\binom{(l+1)\binom{2t+2}{2}}{l},2t+2,s\right).$$
Our choice of $f_{\ref{lem:gettingmixed}}$ and $g_{\ref{lem:gettingmixed}}$ involves the above value $m$ as well as the constants from Lemma~\ref{lem:getample}. Specifically, we claim that:
$$f_{\ref{lem:gettingmixed}}=f_{\ref{lem:gettingmixed}}(d,l,r,s,t)=f_{\ref{lem:getample}}\left(d,(l+1)\binom{2t+2}{2},r,2t+2,m\right)$$
and
$$g_{\ref{lem:gettingmixed}}=g_{\ref{lem:gettingmixed}}(d,l,r,s,t)=g_{\ref{lem:getample}}\left(d,(l+1)\binom{2t+2}{2},r,2t+2,m\right)$$
satisfy the lemma.

Let $\mf{c}$ be a $(f_{\ref{lem:gettingmixed}},g_{\ref{lem:gettingmixed}})$-constellation. Assume that \ref{lem:gettingmixed}\ref{lem:gettingmixed_a} does not hold. By the choice of $f_{\ref{lem:gettingmixed}}$ and $g_{\ref{lem:gettingmixed}}$, we can apply Lemma~\ref{lem:getample} to $\mf{c}$. Note that \ref{lem:getample}\ref{lem:getample_a} implies \ref{lem:gettingmixed}\ref{lem:gettingmixed_a}. So we may assume that \ref{lem:getample}\ref{lem:getample_b} holds, that is, there is a $d$-ample $(m, (l+1)\binom{2t+2}{2})$-constellation $\mf{m}$ which sits $\mf{c}$.

For a $(2t+2)$-subset $T$ of $S_{\mf{m}}$, we say that a path $L\in \mca{L}_{\mf{m}}$ is \textit{$T$-friendly} if for all distinct $x,y\in T$, there is a $\mf{c}$-route $R$ from $x$ to $y$ with $R^*\subseteq L$ such that $T\setminus \{x,y\}$ is anticomplete to $R^*$. We claim that:

     \sta{\label{st:mostpathswin} For every $(2t+2)$-subset $T$ of $S_{\mf{m}}$, the number of $T$-friendly paths in $\mca{L}_{\mf{m}}$ is fewer than $\binom{2t+2}{2}$.}

Suppose not. Then for some $T\subseteq S_{\mf{m}}$ with $|T|=2t+2$, there are at least $\binom{2t+2}{2}$ paths in $\mca{L}_{\mf{m}}$ that are $T$-friendly. In particular, one may choose $\binom{2t+2}{2}$ pairwise distinct $T$-friendly paths $(L_{\{x,y\}}:\{x,y\}\in \binom{T}{2})$ in $\mca{L}_{\mf{m}}$, each assigned to a $2$-subset of $T$. From the definition of a $T$-friendly path, it follows that for every $\{x,y\}\in \binom{T}{2}$, there is a $\mf{c}$-route $R_{\{x,y\}}$ from $x$ to $y$ with $R_{\{x,y\}}^*\subseteq L_{\{x,y\}}$ such that $T\setminus \{x,y\}$ is anticomplete to $R_{\{x,y\}}^*$. But then the subgraph of $\mf{m}$ (and so of $\mf{c}$) induced by 
$$\bigcup_{\{x,y\}\in \binom{T}{2}}R_{\{x,y\}}$$
is isomorphic to a proper subdivision of $K_{2t+2}$, a contradiction to the assumption that $\ref{lem:gettingmixed}\ref{lem:gettingmixed_a}$ does not hold. This proves \eqref{st:mostpathswin}.
\medskip

For the rest of the proof, fix a linear order $\pre$ on $S_{\mf{m}}$. From \eqref{st:mostpathswin}, we deduce that:

\sta{\label{st:mostpathswin2} For every $T=\{x_1,\ldots, x_{2t+2}\}\subseteq S_{\mf{m}}$ where $x_1\prec \cdots\prec x_{2t+2}$, there is a $2$-subset $\{i(T),j(T)\}$ of $\poi_{2t+2}$ with $i(T)<j(T)$ as well as and $l$-subset $\mca{L}(T)$ of $\mca{L}_{\mf{m}}$ for which the following holds. Let $R$ be a $\mf{c}$-route from $x_{i(T)}$ to $x_{j(T)}$ such that $R^*\subseteq V(\mca{L}(T))$. Then some vertex in $T\setminus \{x_{i(T)},x_{j(T)}\}$ has a neighbor in $R$.}

Let $\mca{A}_T$ be the set of all paths in $\mca{L}_{\mf{m}}$ that are not $T$-friendly. Then, for every $L\in \mca{A}_T$, there is a $2$-subset $\{i_{L,T},j_{L,T}\}$ of $\poi_{2t+2}$ with $i_{L,T}<j_{L,T}$ such that for every $\mf{c}$-route $R$ from $x_{i_{L,T}}$ to $x_{j_{L,T}}$ with $R^*\subseteq L$, some vertex in $T\setminus \{x_{i_{L,T}},x_{j_{L,T}}\}$ has a neighbor in $R$. Moreover, by \eqref{st:mostpathswin}, we have $$|\mca{A}_T|>|\mca{L}_{\mf{m}}|-\binom{2t+2}{2}=l\binom{2t+2}{2}.$$ It follows that there exists $\mca{L}(T)\subseteq \mca{A}_T\subseteq \mca{L}_{\mf{m}}$ with $|\mca{L}(T)|=l$ such that for all distinct $L,L'\in \mca{L}(T)$, we have $i_{L,T}=i_{L',T}$ and $j_{L,T}=j_{L',T}$. This proves \eqref{st:mostpathswin2}.
\medskip

Henceforth, for each $(2t+2)$-subset $T$ of $S_{\mf{m}}$, let $\{i(T),j(T)\}\subseteq \poi_{2t+2}$ and $\mca{L}(T)\subseteq \mca{L}_{\mf{m}}$ be as given by \eqref{st:mostpathswin2}. Let 
$$\Phi:\binom{S_{\mf{m}}}{2t+2}\rightarrow \binom{\poi_{2t+2}}{2}\times \binom{\mca{L}_{\mf{m}}}{l}$$
be the map given by 
$\Phi(T)=(\{i(T),j(T)\},\mca{L}(T))$. By the choice of $m$, we can apply Theorem~\ref{thm:multiramsey} and deduce that there exist $i,j\in \poi_{2t+2}$ with $i<j$, an $l$-subset  $\mca{L}$ of  $\mca{L}_{\mf{m}}$ and an $s$-subset $S$ of $S_{\mf{m}}$, such that for every $(2t+2)$-subset $T$ of $S$, we have $i(T)=i, j(T)=j$ and $\mca{L}(T)=\mca{L}$.

Now, consider the $(s,l)$-constellation $\mf{b}=(S,\mca{L})$. We wish to prove that $\mf{b}$ satisfies \ref{lem:gettingmixed}\ref{lem:gettingmixed_b}. Note that $\mf{b}$ sits in $\mf{m}$. In particular, $\mf{b}$ is $d$-ample because $\mf{m}$ is, and $\mf{b}$ sits in $\mf{c}$ because $\mf{m}$ does. It remains to show that $\mf{b}$ is $(2t,2t,2t)$-mixed. To see this, consider the restriction of $\pre$ to $S$. Let $x,y\in S$ with $x\prec y$, let $R$ be $\mf{c}$-route from $x$ to $y$ and let $Q^-, Q,Q^+\subseteq S_{\mf{c}}$ such that $|Q^-|=|Q|=|Q^+|=2t$ and $Q^-\prec x\prec Q \prec y\prec Q^+$. Recall that $i,j\in \poi_{2t+2}$ with $i<j$, and so we may choose $X^-\subseteq Q^-, X\subseteq Q$ and $X^+\subseteq Q^+$ such that $|X^-|=i-1, |X|=j-i-1$ and $|X^+|=2t+2-j$. In particular, we have $X^-\prec x\prec X \prec y\prec X^+$, and
$$T=X^-\cup \{x\}\cup X\cup \{y\}\cup X^+$$
is a $(2t+2)$-subset of $S$. Hence, by the choice of $i,j\in \poi_{2t+2}$, some vertex in $T\setminus \{x,y\}=X^-\cup X\cup X^+\subseteq Q^-\cup Q\cup Q^+$ has a neighbor in $R$. This completes the proof of Lemma~\ref{lem:gettingmixed}.
\end{proof}

\section{From mixed to zigzagged} \label{sec:zigzagged}
In this section, we show that:

\begin{lemma}\label{lem:mixedtozigzag}
   \sloppy For all $l,l',s,s',t\in \poi$ and $\sigma\in \poi\cup \{0\}$, there are constants $f_{\ref{lem:mixedtozigzag}}=f_{\ref{lem:mixedtozigzag}}(l,l',s,s',t, \sigma)\in \poi$ and  $g_{\ref{lem:mixedtozigzag}}=g_{\ref{lem:mixedtozigzag}}(l,l',s,s',t, \sigma)\in \poi$ with the following property. Let $q\in \poi$ and let $q^-,q^+\in \poi\cup \{0\}$ such that $q^-+q^+=\sigma$. Let $\mf{b}$ be an ample and $(q^-,q,q^+)$-mixed $(f_{\ref{lem:mixedtozigzag}},g_{\ref{lem:mixedtozigzag}})$-constellation. Then one of the following holds.
   \begin{enumerate}[\rm (a)]
        \item\label{lem:mixedtozigzag_a} There is an induced subgraph of $\mf{b}$ isomorphic to a proper subdivision of $K_{2t+2}$.
        \item\label{lem:mixedtozigzag_b} There is an interrupted $(s,l)$-constellation which sits in $\mf{b}$. 
    \item\label{lem:mixedtozigzag_c} There is a $q$-zigzagged $(s',l')$-constellation which sits in $\mf{b}$.
    \end{enumerate}
\end{lemma}

The proof of Lemma~\ref{lem:mixedtozigzag} uses a result from \cite{hglemma}, which we state below. By a \textit{hypergraph} we mean a pair $H=(V(H),E(H))$ where $V(H)$ is a finite set, $E(H)\subseteq 2^{V(H)} \setminus \{\varnothing\}$ and $V(H),E(H)\neq \varnothing$. The elements of $H$ are called the \textit{vertices} of $H$ and the elements of $E(H)$ are called the \textit{hyperedges} of $H$. We need to define three parameters associated with a hypergraph $H$:
\begin{enumerate}[1.]
    \item Let $\nu(H)$ be the maximum number of pairwise disjoint hyperedges in $H$.
    \item Let $\tau(H)$ be the minimum cardinality of a set $X\subseteq V(H)$ where $e\cap X\neq \varnothing$ for all $e\in E(H)$.
    \item Let $\lambda(H)$ be the maximum $k\in \poi$ for which there is a $k$-subset $F$ of $E(H)$ with the following property: for every $2$-subset $\{e,e'\}$ of $F$, some vertex $v_{\{e,e'\}}\in V(H)$ satisfies $\{f\in F: v_{\{e,e'\}}\in f\}=\{e,e'\}$.
\end{enumerate}
(Observe that if the hyperedges of $H$ are pairwise disjoint, then we have $\lambda(H)=1$, and otherwise we have $\lambda(H)\geq 2$.)

The following is proved in \cite{hglemma}:
\begin{theorem}[Ding, Seymour and Winkler \cite{hglemma}]\label{thm:hglemma}
    Let $a,a'\in \poi$ and let $H$ be a hypergraph with $\nu(H)\leq a$ and $\lambda(H)\leq a'$. Then we have 
    $$\tau(H)\leq 11a^2(a+a'+3)\binom{a+a'}{a'}^2.$$
\end{theorem}
Theorem \ref{thm:hglemma} is often useful for finding certain induced subgraphs. In our setting, hypergraphs $H$ with $\lambda(H)$ large will translate to induced subdivisions of large complete graphs. 

We now turn to the proof of the main result in this section:

\begin{proof}[Proof of Lemma~\ref{lem:mixedtozigzag}] Let $l,l',s',t\in \poi$ (so all the variables in the statement except for $s$ and $\sigma$) be fixed. First, we define two sequences
$$\{a_{s,\sigma}:s\in \poi, \sigma\in \poi\cup \{0\}\}$$ and 
$$\{b_{s,\sigma}:s\in \poi, \sigma\in \poi\cup \{0\}\}$$
of positive integers recursively, as follows. For every $\sigma\in \poi\cup \{0\}$, let 
$$a_{1,\sigma}=1; \quad b_{1,\sigma}=l.$$
For every $s\in \poi$, let
$$a_{s,0}=s'; \quad b_{s,0}=l'.$$
For all $s\geq 2$ and $\sigma\geq 1$, assuming $a_{s-1,\sigma}, a_{s,\sigma-1}, b_{s-1,\sigma}$ and $b_{s,\sigma-1}$ are all defined, let
$$a_{s,\sigma}=11 a^2_{s-1,\sigma}(a_{s-1,\sigma}+2t+4)\binom{a_{s-1,\sigma}+2t+1}{2t+1}^2a_{s,\sigma-1}$$
and let
$$b_{s,\sigma}=\binom{a_{s,\sigma}}{a_{s-1,\sigma}}b_{s-1,\sigma}+\binom{a_{s,\sigma}}{a_{s,\sigma-1}}b_{s,\sigma-1}.$$
This concludes the definition of the two sequences. Back to the proof of \ref{lem:mixedtozigzag}, we will prove by induction on $s+\sigma$ that
$$f_{\ref{lem:mixedtozigzag}}=f_{\ref{lem:mixedtozigzag}}(l,l',s,s',t, \sigma)=a_{s,\sigma}+1$$
and
$$g_{\ref{lem:mixedtozigzag}}=g_{\ref{lem:mixedtozigzag}}(l,l',s,s',t, \sigma)=b_{s,\sigma}$$

satisfy the lemma.
\medskip

Let $q\in \poi$ and let $q^-,q^+\in \poi\cup \{0\}$ such that $q^-+q^+=\sigma$. Let $\mf{b}$ be an ample and $(q^-,q,q^+)$-mixed $(f_{\ref{lem:mixedtozigzag}},g_{\ref{lem:mixedtozigzag}})$-constellation. Assume that $s=1$. From the choice of $a_{1,\sigma}$ and $b_{1,\sigma}$, it follows that  $\mf{b}$ is a $(1,l)$-constellation, and so $\mf{b}$ is interrupted, implying that \ref{lem:mixedtozigzag}\ref{lem:gettingmixed_b} holds. Next, assume that $\sigma=0$. Then we have $q^-=q^+=0$. From the choice of $a_{s,0}$ and $b_{s,0}$, it follows that $\mf{b}$ is an $(s',l')$-constellation which is $(0,q,0)$-mixed, and so $\mf{b}$ is $q$-zigzagged, implying that \ref{lem:mixedtozigzag}\ref{lem:gettingmixed_b} holds.
\medskip

We may assume from now on that $s\geq 2$ and $\sigma\geq 1$. Let $\pre$ be the linear order on $S_{\mf{b}}$ with respect to which $\mf{b}$ is $(q^-,q,q^+)$-mixed. Since $q^-+q^+=\sigma\geq 1$, it follows that either $q^-\geq 1$ or $q^+\geq 1$. Up to the reversal of $\pre$, we may assume that $q^+\geq 1$. Let $v$ be the unique vertex in $S_{\mf{b}}$ with $S_{\mf{b}}\setminus \{v\}\prec v$.
\medskip

For $L\in \mca{L}_{\mf{b}}$, we denote by $N_L(v)$ the set of all neighbors of $v$ in $L$, and by $\mca{K}_L$ the set of all components of $L\setminus N_L(v)$.
Observe that $\mca{K}_L$ is a set of pairwise disjoint and anticomplete paths in $\mf{b}$ with $|\mca{K}_L|\geq 2$, and that $v$ is anticomplete to $V(\mca{K}_L)$ in $\mf{b}$. For each $L\in \mca{L}_{\mf{b}}$ and every $x\in S_{\mf{b}}\setminus \{v\}$, let $\mca{K}_{x,L}$ be the set of all paths $K\in \mca{K}_{L}$ such that $x$ has a neighbor in $K$ (note that $N_L(x)\subseteq V(\mca{K}_L)$ because $\mf{b}$ is ample). For each $L\in \mca{L}_{\mf{b}}$, we define the hypergraph $H_L$ with $V(H_L)=\mca{K}_L$ and $E(H_L)=\{\mca{K}_{x,L}: x\in S_{\mf{b}}\setminus \{v\}\}$. 
\medskip

Let
$$\mca{M}=\{L\in \mca{L}_{\mf{b}}:\nu(H_L)\geq a_{s-1,\sigma}\}$$
and let
$$\mca{N}=\{L\in \mca{L}_{\mf{b}}:\nu(H_L)\leq a_{s-1,\sigma}\}.$$
Since $\mca{L}_{\mf{b}}=\mca{M}\cup \mca{N}$, it follows from the choice of $b_{s,\sigma}$ that either 

$$|\mca{M}|\geq \binom{a_{s,\sigma}}{a_{s-1,\sigma}}b_{s-1,\sigma}$$
or
$$|\mca{N}|\geq \binom{a_{s,\sigma}}{a_{s,\sigma-1}}b_{s,\sigma-1}.$$

Assume that the former holds. Then, since $|S_{\mf{b}}\setminus \{v\}|=a_{s,\sigma}$ and from the definition of $\mca{M}$, it follows that there is an $a_{s-1,\sigma}$-subset
$S_1$ of $S_{\mf{b}}\setminus \{v\}$ as well as a $b_{s-1,\sigma}$-subset $\mca{L}_1$ of $\mca{M}\subseteq \mca{L}_{\mf{b}}$ such that for every $L\in \mca{L}_1$, the sets $(\mca{K}_{x,L}:x\in S_1)$ are pairwise disjoint. Let $\mf{b}_1$ be the $(a_{s-1,\sigma},b_{s-1,\sigma})$-constellation $(S_1,\mca{L}_1)$. Then $\mf{b}_1$ sits in $\mf{b}$; in particular, $\mf{b}_1$ is ample and $(q^-,q,q^+)$-mixed, because $\mf{b}$ is. Therefore, by the induction hypothesis applied to $\mf{b}_1$, one of the following holds:
 \begin{enumerate}[\rm (A{1})]
        \item\label{lem:mixedtozigzag_a1} There is an induced subgraph of $\mf{b}_1$ isomorphic to a proper subdivision of $K_{2t+2}$.
        \item\label{lem:mixedtozigzag_b1} There is an interrupted $(s-1,l)$-constellation which sits in $\mf{b}_1$. 
    \item\label{lem:mixedtozigzag_c1} There is a $q$-zigzagged $(s',l')$-constellation which sits in $\mf{b}_1$.
    \end{enumerate}
Since $\mf{b}_1$ sits in $\mf{b}$, it follows that \ref{lem:mixedtozigzag_a1} and \ref{lem:mixedtozigzag_c1} imply \ref{lem:mixedtozigzag}\ref{lem:mixedtozigzag_a} and \ref{lem:mixedtozigzag}\ref{lem:mixedtozigzag_c}, respectively. So we may assume that \ref{lem:mixedtozigzag_b1} holds, that is, there is an interrupted $(s-1,l)$-constellation $\mf{a}$ which sits in $\mf{b}_1$. Let $\pre_1$ be the linear order on $S_{\mf{a}}$ with respect to which $\mf{a}$ is interrupted. Recall that by the choice of $S_1$ and $\mca{L}_1$, the vertex $v$ has a neighbor in every $\mf{b}_1$-route. Since $\mf{a}$ sits in $\mf{b}_1$, it follows that $v$ has a neighbor in every $\mf{a}$-route. Consequently, $(S_{\mf{a}}\cup \{v\},\mca{L}_{\mf{a}})$ is an $(s,l)$-constellation which sits in $\mf{b}$, and extending $\pre_1$ from $S_{\mf{a}}$ to $S_{\mf{a}}\cup \{v\}$ by setting $S_{\mf{a}}\prec_1 \{v\}$, it follows that $(S_{\mf{a}}\cup \{v\},\mca{L}_{\mf{a}})$ is interrupted. But then \ref{lem:mixedtozigzag}\ref{lem:mixedtozigzag_b} holds. This completes the induction in the case
$|\mca{M}|\geq \binom{a_{s,\sigma}}{a_{s-1,\sigma}}b_{s-1,\sigma}$.
\medskip

Henceforth, assume that $|\mca{N}|\geq \binom{a_{s,\sigma}}{a_{s,\sigma-1}}b_{s,\sigma-1}$. As it turns out, we may also assume that $\lambda(H_L)\leq 2t+1$ for every $L\in \mca{N}$. In fact, we have:

\sta{\label{st:boundedlambda} Suppose that $\lambda(H_L)\geq 2t+2$ for some $L\in \mca{L}_{\mf{b}}$. Then \ref{lem:mixedtozigzag}\ref{lem:mixedtozigzag_a} holds.}

By the definition of the hypergraph $H_L$ and the parameter $\lambda$, there is a $(2t+2)$-subset $X$ of $S_{\mf{b}}\setminus \{v\}$ with the following property: for every $2$-subset $\{x,y\}$ of $X$, there exists $K_{\{x,y\}}\in \mca{K}_L$ such that $x$ and $y$ both have neighbors in $K_{\{x,y\}}$ while $X\setminus \{x,y\}$ and $K_{\{x,y\}}$ are anticomplete in $\mf{b}$. In particular, for every $2$-subset $\{x,y\}$ of $X$, there is a $\mf{b}$-route $R_{\{x,y\}}$ from $x$ to $y$ with $R^*\subseteq K_{\{x,y\}}$. But now the subgraph of $\mf{b}$ induced by 
$$\bigcup_{\{x,y\}\in \binom{X}{2}} R_{\{x,y\}}$$
is isomorphic to a proper subdivision of $K_{2t+2}$. This proves \eqref{st:boundedlambda}.
\medskip

From \eqref{st:boundedlambda}, the definition of $\mca{N}$, and Theorem~\ref{thm:hglemma}, it follows that for every $L\in \mca{N}$, we have
$$\tau(H_L)\leq 11 a^2_{s-1,\sigma}(a_{s-1,\sigma}+2t+4)\binom{a_{s-1,\sigma}+2t+1}{2t+1}^2;$$
which, combined with the choice of $a_{s,\sigma}$, yields $$|S_{\mf{b}}\setminus \{v\}|=a_{s,\sigma}\geq \tau(H_L)a_{s,\sigma-1}.$$

In particular, for every $L\in \mca{N}$, there is an $a_{s,\sigma-1}$-subset
$S_{2,L}$ of $S_{\mf{b}}\setminus \{v\}$ as well as a path $K_L\in \mca{K}_{L}$ such that 
$$K_{L}\in \bigcap_{x\in S_{2,L}}\mca{K}_{x,L};$$
that is, every vertex in $S_{2,L}$ has a neighbor in $K_L$. This, along with the assumption that $|\mca{N}|\geq \binom{a_{s,\sigma}}{a_{s,\sigma-1}}b_{s,\sigma-1}$, implies that there is an $a_{s,\sigma-1}$-subset
$S_2$ of $S_{\mf{b}}\setminus \{v\}$ and a $b_{s,\sigma-1}$-subset $\mca{L}_2$ of $\mca{N}\subseteq \mca{L}_{\mf{b}}$ such that for every $L\in \mca{L}_2$, some path $K_L\in \mca{K}_{L}$ satisfies
$$K_{L}\in \bigcap_{x\in S_{2}}\mca{K}_{x,L}.$$
Again, this means in words that for every $L\in \mca{L}$, there is a path $K_L\in \mca{K}_{L}$ such that every vertex in $S_2$ has a neighbor in $K_L$. We deduce that $\mf{b}_2=(S_2,\{K_L:L\in \mca{L}_2\})$ is a $(a_{s,\sigma-1},b_{s,\sigma-1})$-constellation. 

Now, observe that $\mf{b}_2$ sits in $\mf{b}$, and so $\mf{b}_2$ is ample because $\mf{b}$ is. Moreover, recall that $\mf{b}$ is $(q^-,q,q^+)$-mixed with respect to the linear order $\pre$ on $S_{\mf{b}}$, where $q^+\geq 1$. Since $S_2\subseteq S_{\mf{b}}\setminus \{v\}\prec v$ and since $v$ is anticomplete to $V(\mca{L}_2)$, it follows that $\mf{b}_2$ is $(q^-,q,q^+-1)$-mixed with respect to the restriction of $\pre$ to $S_{2}$. Therefore, since $q^-+(q^+-1)=\sigma-1$, it follows from the induction hypothesis applied to the constellation $\mf{b}_2$ that one of the following holds.
 \begin{enumerate}[\rm (A{2})]
        \item\label{lem:mixedtozigzag_a2} There is an induced subgraph of $\mf{b}_2$ isomorphic to a proper subdivision of $K_{2t+2}$.
        \item\label{lem:mixedtozigzag_b2} There is an interrupted $(s,l)$-constellation which sits in $\mf{b}_1$. 
    \item\label{lem:mixedtozigzag_c2} There is a $q$-zigzagged $(s',l')$-constellation which sits in $\mf{b}_2$.
    \end{enumerate}
Since $\mf{b}_2$ sits in $\mf{b}$, it follows that \ref{lem:mixedtozigzag_a2} implies \ref{lem:mixedtozigzag}\ref{lem:mixedtozigzag_a}, \ref{lem:mixedtozigzag_b2} implies \ref{lem:mixedtozigzag}\ref{lem:mixedtozigzag_b} and \ref{lem:mixedtozigzag_c2} implies \ref{lem:mixedtozigzag}\ref{lem:mixedtozigzag_c}. This completes the induction in the case
$|\mca{N}|\geq \binom{a_{s,\sigma}}{a_{s,\sigma-1}}b_{s,\sigma-1}$, hence completing the proof of Lemma~\ref{lem:mixedtozigzag}.
\end{proof}

\section{Finishing the proof of Theorem~\ref{thm:mainconstellation}} \label{sec:finale1}
Finally, let us put everything together and prove our first main result, which we restate: 

\mainconstellation*
\begin{proof}
    Let
    $$f=f_{\ref{lem:mixedtozigzag}}(l,l',s,s',t, 4t)$$ 
    and let
    $$g=g_{\ref{lem:mixedtozigzag}}(l,l',s,s',t, 4t).$$
We claim that 
$$f_{\ref{thm:mainconstellation}}=f_{\ref{thm:mainconstellation}}(d,l,l',r,s,s',t)=f_{\ref{lem:gettingmixed}}(d,g,r,f,t)$$
and $$g_{\ref{thm:mainconstellation}}=g_{\ref{thm:mainconstellation}}(d,l,l',r,s,s',t)=g_{\ref{lem:gettingmixed}}(d,g,r,f,t)$$
satisfy the theorem.

Let $\mf{c}$ be a $(f_{\ref{thm:mainconstellation}},g_{\ref{thm:mainconstellation}})$-constellation. By the choice of $f_{\ref{thm:mainconstellation}}$ and $g_{\ref{thm:mainconstellation}}$, we can apply Lemma~\ref{lem:gettingmixed} to $\mf{c}$ and deduce that either \ref{lem:gettingmixed}\ref{lem:gettingmixed_a} or \ref{lem:gettingmixed}\ref{lem:gettingmixed_b} holds. The former is identical to \ref{thm:mainconstellation}\ref{thm:mainconstellation_a}, so we may assume that the latter outcome holds, that is, there is a $d$-ample and $(2t,2t,2t)$-mixed $(f,g)$-constellation $\mf{b}$ which sits in $\mf{c}$.

In particular, $\mf{b}$ is ample. Therefore, by the choices of $f$ and $g$, we can apply Lemma~\ref{lem:mixedtozigzag} to $\mf{b}$. Again, note that  \ref{lem:mixedtozigzag}\ref{lem:mixedtozigzag_a} is identical to \ref{thm:mainconstellation}\ref{thm:mainconstellation_a}. Also, since $\mf{b}$ is $d$-ample, it follows that the constellations given by both \ref{lem:mixedtozigzag}\ref{lem:mixedtozigzag_b} and \ref{lem:mixedtozigzag}\ref{lem:mixedtozigzag_c} are also be $d$-ample. Hence, \ref{lem:mixedtozigzag}\ref{lem:mixedtozigzag_b} implies  \ref{thm:mainconstellation}\ref{thm:mainconstellation_b} and \ref{lem:mixedtozigzag}\ref{lem:mixedtozigzag_c} implies  \ref{thm:mainconstellation}\ref{thm:mainconstellation_c}. This completes the proof of Theorem~\ref{thm:mainconstellation}.
\end{proof}

\section{Obtaining a constellation} \label{sec:bip}

In this section, we prove Theorem \ref{thm:main_ind_minor}. This, together with Theorem \ref{thm:mainconstellation}, implies Theorem \ref{thm:motherKtt}. For ease of notation, we define the following. Let $G$ be a graph and let $s, t \in \mathbb{N}$. By an \emph{induced $(s,t)$-model in $G$} we mean an $(s+t)$-tuple $M=(A_1,\ldots, A_s; B_1,\ldots, B_t)$ of pairwise disjoint connected induced subgraphs of $G$ such that:
\begin{itemize}
    \item $A_1, \dots, A_s$ are pairwise anticomplete in $G$; 
    \item $B_1, \dots, B_t$ are pairwise anticomplete in $G$; and
    \item For all $i \in \poi_s$ and $j \in \poi_t$, the sets $A_i$ and $B_j$ are not anticomplete in $G$. 
\end{itemize}
 It is readily observed that a graph $G$ has an induced minor isomorphic to $K_{s,t}$ if and only if there is an induced $(s,t)$-model in $G$.

Let $G$ be a graph and let $M=(A_1,\ldots, A_s; B_1,\ldots, B_t)$ be an induced $(s,t)$-model in $G$. We call the sets $A_1, \dots, A_s, B_1, \dots, B_t$ the \emph{branch sets} of $M$. We also define $$A(M)=\bigcup_{i=1}^sA_i$$ and $$B(M)=\bigcup_{j=1}^tB_j.$$
We say that $M$ is \emph{$A$-linear} if $A_i$ is a path in $G$ for every $i\in \poi_{s}$. Similarly, we say that $M$ is  \emph{$B$-linear} if $B_j$ is a path in $G$ for every $j\in \poi_{t}$. We say $M$ is \emph{linear} if it is both $A$-linear and $B$-linear. 

The proof of Theorem \ref{thm:main_ind_minor} has three steps: First, in Theorems~\ref{thm:linear} and \ref{thm:2linear}, we reduce the problem to the linear case. Next, in Lemma \ref{lem:getample2}, we further simplify the setup using Lemma \ref{lem:getample}. Finally, in Theorem \ref{thm:mainktt}, we finish the proof; this is only part of the argument in which we exclude $W_{r \times r}$ as an induced minor, rather than a 1-subdivision of a complete graph. 

\begin{theorem} \label{thm:linear}
    \sloppy For all $r, s, t\in \poi$, there are constants $f_{\ref{thm:linear}}=f_{\ref{thm:linear}}(r, s, t)\in \poi$ and $g_{\ref{thm:linear}}=g_{\ref{thm:linear}}(r,s,t)\in \poi$ with the following property. Let $G$ be a graph and let $M=(A_1, \dots, A_{f_{\ref{thm:linear}}}; B_1, \dots, B_{g_{\ref{thm:linear}}})$ be an induced $(f_{\ref{thm:linear}}, g_{\ref{thm:linear}})$-model in $G$. Then one of the following holds.
    \begin{enumerate}[\rm (a)]
        \item\label{thm:linear_a} There is an $A$-linear induced $(s, t)$-model $M'=(A_1', \dots, A_s'; B_1', \dots, B_t')$ in $G$ such that for every $i \in \poi_t$, we have $B_i' \in \{B_1, \dots, B_{g_{\ref{thm:linear}}}\}$.
        \item\label{thm:linear_b} There is an induced minor of $G$ isomorphic to a $1$-subdivision of $K_r$.
    \end{enumerate} 
\end{theorem}
\begin{proof}
    Let $$g_{\ref{thm:linear}}=g_{\ref{thm:linear}}(r,s,t) = 11(r+3)r^2t.$$
    Let $q = s{g_{\ref{thm:linear}} \choose t}$ and let $$f_{\ref{thm:linear}}=f_{\ref{thm:linear}}(r,s,t) = q{g_{\ref{thm:linear}} \choose 2}.$$

    Let $G$ be a graph and let $M=(A_1, \dots, A_{f_{\ref{thm:linear}}}; B_1, \dots, B_{g_{\ref{thm:linear}}})$ be an induced $(f_{\ref{thm:linear}}, g_{\ref{thm:linear}})$-model in $G$. Suppose for a contradiction that neither \ref{thm:linear}\ref{thm:linear_a} nor \ref{thm:linear}\ref{thm:linear_b} holds. 
    We first partition $\{1, \dots, f_{\ref{thm:linear}}\}$ into ${g_{\ref{thm:linear}} \choose 2}$ sets $(X_{i, j}: i,j\in \poi_{g_{\ref{thm:linear}}}, i<j)$, each of cardinality $q$. 

    Fix $i, j \in \poi_{g_{\ref{thm:linear}}}$ with $i < j$. For every $k \in X_{i, j}$, we let $P_{i,j,k}$ be a path in $A_k$ with $|P_k|$ minimum such that $P_k$ contains a vertex with a neighbor in $B_i$ and a vertex with a neighbor in $B_j$. Note that such a choice of $P_{i,j,k}$ is possible because $A_k$ is connected and contains both a vertex with a neighbor in $B_i$ and a vertex with a neighbor in $B_j$. For every $k \in X_{i, j}$, let $$S_{i,j,k} = \{l \in \poi_{g_{\ref{thm:linear}}} : P_{i,j,k} \textnormal{ is not anticomplete to } B_l\}.$$ It follows that $\{i, j\} \subseteq S_{i,j,k}$. Moreover, we deduce that:

    \sta{\label{st:bigs} For all $i, j \in \poi_{g_{\ref{thm:linear}}}$ with $i < j$, there exists $k_{i, j} \in X_{i,j}$ such that $|S_{i,j,k_{i,j}}| < t$.}

    For suppose there exist $i, j \in \poi_{g_{\ref{thm:linear}}}$ with $i < j$ such that $|S_k| \geq t$ for all $k \in X_{i, j}$. For each $k \in X_{i, j}$, choose a $t$-subset $T_k$ of $S_k$. From the choice of $q = |X_{i,j}|$, it follows that there is an $s$-subset $Y$ of $X_{i,j}$ such that $T_k = T_{k'}$ for all $k, k' \in Y$. Let $T$ be the set with $T = T_k$ for all $k \in Y$. But now $(P_k:k \in Y ; B_l: l \in T)$ is an $A$-linear induced $(s,t)$-model in $G$ that satisfies \ref{thm:linear}\ref{thm:linear_a}, a contradiction. This proves \eqref{st:bigs}. 

    \medskip

    From now on, for each choice of $i, j \in \poi_{g_{\ref{thm:linear}}}$ with $i<j$, let $k_{i, j} \in X_{i,j}$ be as given by \eqref{st:bigs}, and write $S_{i,j}=S_{i,j,k_{i,j}}$ and $P_{i,j}=P_{i,j,k_{i,j}}$. 
    
    Let $H$ be the hypergraph defined as follows. For each choice of $i, j \in \poi_{g_{\ref{thm:linear}}}$ with $i < j$, the hypergraph $H$ has a vertex $v_{i, j}$, and for every $l \in \poi_{g_{\ref{thm:linear}}}$, the hypergraph $H$ has a hyperedge $e_l = \{v_{i,j} : l \in S_{i, j}\}$ (so $H$ has ${g_{\ref{thm:linear}} \choose 2}$ vertices and $g_{\ref{thm:linear}}$ hyperedges).
    
    We consider the values of $\tau(H), \nu(H)$ and $\lambda(H)$. Since $|S_{i,j}| < t$ for all $i,j\in \poi_{g_{\ref{thm:linear}}}$ with $i<j$, it follows that every vertex of $H$ belongs to fewer than $t$ hyperedges. In particular, we have $\tau(H) >g_{\ref{thm:linear}}/t$, which along with the choice of $g_{\ref{thm:linear}}$ implies that:

    \sta{\label{st:hypertau} We have $\tau(H)>11(r+3)r^2$.}
    
    Also, recall that for all $i,j \in \poi_{g_{\ref{thm:linear}}}$ with $i < j$, we have $i,j\in S_{i,j}$, and so $v_{i, j} \in e_i \cap e_j$. In particular, 

    \sta{\label{st:hypernu} We have $\nu(H)=1$.}

Let us now show that:

    \sta{\label{st:hyper1} We have $\lambda(H) < r$.}
  
  Suppose not; that is, there is an $r$-subset $F$ of $\poi_{g_{\ref{thm:linear}}}$ such that for each choice of $i,j\in F$ with $i<j$, there exist $\alpha_{i,j},\beta_{i,j}\in \poi_{g_{\ref{thm:linear}}}$ with $\alpha_{i,j}<\beta_{i,j}$ such that 
    $$\{l \in F: v_{\alpha_{i,j},\beta_{i,j}}\in e_l\} = \{i, j\}.$$
By the definition of $H$, we obtain that
    $$F\cap S_{\alpha_{i,j},\beta_{i,j}}=\{i,j\}.$$
     For each choice of $i,j\in F$ with $i<j$, it follows from the definition of $S_{\alpha_{i,j},\beta_{i,j}}$ that $B_i$ and $B_j$ are the only sets among $(B_l: l \in F)$ to which the path $P_{\alpha_{i,j},\beta_{i,j}}$ is not anticomplete in $G$. Moreover, recall that the sets $(B_l: l \in F)$ are connected and pairwise anticomplete in $G$, and the sets $(P_{\alpha_{i,j},\beta_{i,j}}: i,j\in F, i<j)$ are connected and pairwise anticomplete in $G$. Therefore, $G$ has an induced minor isomorphic to a $1$-subdivision of $K_r$, where the sets $(B_l: l \in F)$ correspond to the vertices of $K_r$ and sets $(P_{\alpha_{i,j},\beta_{i,j}}: i,j\in F, i<j)$ correspond to the ${r\choose 2}$ vertices obtained by $1$-subdividing each edge of $K_r$. But now \ref{thm:linear}\ref{thm:linear_b} holds, a contradiction. This proves \eqref{st:hyper1}. 

    \medskip
    
    From \eqref{st:hypernu}, \eqref{st:hyper1} and Theorem \ref{thm:hglemma} (plugging in $a = 1$ and $a' = r-1$), it follows that
    $$\tau(H)\leq 11a^2(a+a'+3)\binom{a+a'}{a'}^2 \leq 11(r+3)r^2.$$
    This is a contradiction to \eqref{st:hypertau}, and concludes the proof of Theorem~\ref{thm:linear}. 
\end{proof}

Let $s,t\in \poi$, let $G$ be a graph and let  $M=(A_1,\ldots, A_s; B_1,\ldots, B_t)$ be an induced $(s,t)$-model in $G$. We denote by $M^T$ the induced $(t,s)$-model $(B_1,\ldots, B_t; A_1,\ldots, A_s)$ in $G$. 

Applying Theorem \ref{thm:linear} to both sides of the bipartition, we deduce the following:

\begin{theorem} \label{thm:2linear}
    For all $r, s, t\in \poi$, there are constants $f_{\ref{thm:2linear}}=f_{\ref{thm:2linear}}(r, s, t)\in \poi$ and $g_{\ref{thm:2linear}}=g_{\ref{thm:2linear}}(r, s, t)\in \poi$ with the following property.  Let $G$ be a graph and assume that there is an induced $(f_{\ref{thm:2linear}}, g_{\ref{thm:2linear}})$-model $M$ in $G$. Then one of the following holds.
     \begin{enumerate}[\rm (a)]
        \item\label{thm:2linear_a} There is a linear induced $(s, t)$-model in $G$.
        \item\label{thm:2linear_b} There is an induced minor of $G$ isomorphic to a $1$-subdivision of $K_r$.
    \end{enumerate} 
\end{theorem}

\begin{proof}
    Let $f_{\ref{thm:linear}}=f_{\ref{thm:linear}}(r, s, t)$ and $g_{\ref{thm:linear}}=g_{\ref{thm:linear}}(r,t)$ be as given by Theorem~\ref{thm:linear}. Let
    $$f_{\ref{thm:2linear}} = f_{\ref{thm:2linear}}(r, s, t) = f_{\ref{thm:linear}}(r, g_{\ref{thm:linear}}, f_{\ref{thm:linear}})$$
    and let
    $$g_{\ref{thm:2linear}} = g_{\ref{thm:2linear}}(r, s, t) = g_{\ref{thm:linear}}(r, g_{\ref{thm:linear}}, f_{\ref{thm:linear}}).$$

    We apply Theorem \ref{thm:linear} to the induced $(f_{\ref{thm:2linear}}, g_{\ref{thm:2linear}})$-model $M$ in $G$. Note that \ref{thm:linear}\ref{thm:linear_b} and \ref{thm:2linear}\ref{thm:2linear_b} are identical, and so we may assume that \ref{thm:linear}\ref{thm:linear_a} holds. In particular, there is an $A$-linear induced $(g_{\ref{thm:linear}}, f_{\ref{thm:linear}})$-model $M'=(A'_1, \dots, A'_{g_{\ref{thm:linear}}}; B'_1, \dots, B'_{f_{\ref{thm:linear}}})$ in $G$. We now apply Theorem \ref{thm:linear} to the induced $(f_{\ref{thm:linear}}, g_{\ref{thm:linear}})$-model  $M'^T=(B'_1, \dots, B'_{f_{\ref{thm:linear}}}; A'_1, \dots, A'_{g_{\ref{thm:linear}}})$ in $G$. Again, since \ref{thm:linear}\ref{thm:linear_b} and \ref{thm:2linear}\ref{thm:2linear_b} are identical, we may assume that \ref{thm:linear}\ref{thm:linear_a} holds; that is, there is an $A$-linear induced $(s,t)$-model $M''=(A_1'', \dots, A_s''; B_1'', \dots, B_t')$ in $G$ such that for all $i \in \poi_{t}$, we have $B_i'' \in \{A_1, \dots, A_{g_{\ref{thm:linear}}}\}$. Since $A_1, \dots, A_{g_{\ref{thm:linear}}}$ are paths in $G$ (as $M'$ is $A$-linear), it follows that $M''$ is $B$-linear. Hence,  $M''$ is both $A$-linear and $B$-linear, and so \ref{thm:2linear}\ref{thm:2linear_a} holds, as desired.
\end{proof}

Let $s, t \in \mathbb{N}$ and let $G$ be a graph. Let $M = (A_1, \dots, A_s; B_1, \dots, B_t)$ be a linear induced $(s,t)$-model in $G$. By contracting each set $A_i$ to a vertex $a_i$, we obtain an induced minor of $G$ which is an $(s, t)$-constellation. We call this constellation the \emph{$A$-contraction of $M$}, and denote it by $\mf{c}(M)$. For convenience, we write 
$S_M$ for $S_{\mf{c}(M)}=\{a_i : i \in \poi_s\}$ and $\mca{L}_M$ for  $\mca{L}_{\mf{c}(M)}=\{B_j: j \in \poi_t\}$; so $\mf{c}(M)=(S_M,\mca{L}_M)$. The \emph{$B$-contraction of $M$} is defined as the $A$-contraction $\mf{c}(M^T)$ of $M^T$. Analogous to the notion for constellations, for $d \in \mathbb{N}$, we say that $M$ is \emph{$d$-$A$-ample} if $\mf{c}(M)$ is $d$-ample, and $M$ is \emph{$d$-$B$-ample} if $\mf{c}(M^T)$ is $d$-ample. We say that $M$ is \emph{$d$-ample} if $M$ is both $d$-$A$-ample and $d$-$B$-ample.

\begin{lemma}\label{lem:getample2}
For all $d,l,r,r's\in \poi$, there are constants $f_{\ref{lem:getample2}}=f_{\ref{lem:getample2}}(d,l,r,r',s)\in \poi$ and $g_{\ref{lem:getample2}}=g_{\ref{lem:getample2}}(d,l,r,r',s)\in \poi$ with the following property. Let $G$ be a graph and let $M = (A_1, \dots, A_{f_{\ref{lem:getample2}}}; B_1, \dots, B_{g_{\ref{lem:getample2}}})$ be a linear induced $(f_{\ref{lem:getample2}},g_{\ref{lem:getample2}})$-model in $G$. Then one of the following holds.
 \begin{enumerate}[\rm (a)]
    \item\label{lem:getample2_a} There is an $(r, r)$-constellation in $G$.  
    \item\label{lem:getample2_b} There is an induced minor of $G$ which is isomorphic to a $1$-subdivision of $K_{r'}$.
    \item\label{lem:getample2_c} There exist subsets $X\subseteq \poi_{f_{\ref{lem:getample2}}}$ and $Y\subseteq \poi_{g_{\ref{lem:getample2}}}$ such that $|X|=s$, $|Y|=l$ and the induced $(s,l)$-model $(A_i: i \in X; B_j: j \in Y)$ in $G$ is $d$-ample.   \end{enumerate}
\end{lemma}
\begin{proof}
   Let 
   $$f_{\ref{lem:getample}}=f_{\ref{lem:getample}}(d,l,\max\{r,3\},r',s)$$
   and
   $$g_{\ref{lem:getample}}=g_{\ref{lem:getample}}(d,l,\max\{r,3\},r',s)$$
   be as given by Lemma~\ref{lem:getample}.
   Let
   $$f_{\ref{lem:getample2}}=f_{\ref{lem:getample2}}(d,l,r,r',s)=f_{\ref{lem:getample}}(d,f_{\ref{lem:getample}},\max\{r,3\},r',g_{\ref{lem:getample}}).$$
and
$$g_{\ref{lem:getample2}}=g_{\ref{lem:getample2}}(d,l,r,r',s)=g_{\ref{lem:getample}}(d,f_{\ref{lem:getample}},\max\{r,3\},r',g_{\ref{lem:getample}})$$
 Let $G$ be a graph and let $M = (A_1, \dots, A_{f_{\ref{lem:getample2}}}; B_1, \dots, B_{g_{\ref{lem:getample2}}})$ be a linear induced $(f_{\ref{lem:getample2}},g_{\ref{lem:getample2}})$-model in $G$. Suppose that \ref{lem:getample2}\ref{lem:getample2_a} and \ref{lem:getample2}\ref{lem:getample2_b} do not hold. We show that:

\sta{\label{st:ktt} Neither $\mf{c}(M)$ nor $\mf{c}(M^T)$ has an induced subgraph isomorphic to $K_{\max\{r,3\}, \max\{r,3\}}$.}

Suppose not. By symmetry, we may assume $\mf{c}(M)$ has an induced subgraph isomorphic to $K_{\max\{r,3\}, \max\{r,3\}}$. Then there are disjoint stable sets $Q, R \subseteq V(\mf{c}(M))$ in $\mf{c}(M)$ such that $|Q| = |R| = \max\{r,3\}$, and every vertex in $Q$ is adjacent to every vertex in $R$ in $\mf{c}(M)$. Since $S_M$ is a stable set in $\mf{c}(M)$, it follows that either $Q$ or $R$ is disjoint from $S_M$; without loss of generality, assume that $R \cap S_M = \varnothing$. Thus, we have $R\subseteq B(M)$; therefore, $R\subseteq B(M)$ is a stable set of cardinality $\max\{r,3\}$ in $G$.

Next, we claim that $Q \subseteq S_M$. For suppose there is a vertex $q \in Q \setminus S_M\subseteq \mf{c}(M) \setminus S_M$. Then, since $R\subseteq \mf{c}(M) \setminus S_M$ and since $q$ is complete to $R$ in $\mf{c}(M)$, it follows that $q$ is a vertex of degree at least $|R|=\max\{r,3\}\geq 3$ in the graph $\mf{c}(M) \setminus S_M$. However, each component of the graph $\mf{c}(M) \setminus S_M$ is a path and so $\mf{c}(M) \setminus S_M$ contains no vertex of degree greater than 2, a contradiction. The claim follows.

Now, for each $i\in \poi_{f_{\ref{lem:getample2}}}$, let $A_i$ be the vertex of $\mf{c}(M)$ obtained from contracting $A_i$. Let $I = \{i\in \poi_{f_{\ref{lem:getample2}}}: a_i \in Q\}$. Then, by the above claim, we have $|I|=|Q|=\max\{r,3\}$. Moreover, for every $i\in I$, every vertex in $R$ has a neighbour in the path $A_i$. But now $(R, \{A_i : i \in I\})$ is a $(\max\{r,3\}, \max\{r,3\})$-constellation in $G$, and so \ref{lem:getample2}\ref{lem:getample2_a} holds, a contradiction. This proves \eqref{st:ktt}. 

\sta{\label{st:ktt2} Neither $\mf{c}(M)$ nor $\mf{c}(M^T)$ has an induced subgraph isomorphic to a proper subdivision of $K_{r'}$.}

Suppose not. By symmetry, we may assume that $\mf{c}(M)$ has an induced subgraph isomorphic to a proper subdivision of $K_{r'}$. Since $\mf{c}(M)$ is an induced minor of $G$, it follows that $G$ has an induced minor isomorphic to a proper subdivision of $K_{r'}$, which in turn implies that $G$ has an induced minor isomorphic to a $1$-subdivision of $K_{r'}$. But then \ref{lem:getample2}\ref{lem:getample2_b} holds, a contradiction. This proves \eqref{st:ktt2}. 
\medskip

We apply Lemma \ref{lem:getample} to the constellation $\mf{c}(M) = (S_M, \mca{L}_M)$. From \eqref{st:ktt} and \eqref{st:ktt2}, it follows that the outcome \ref{lem:getample}\ref{lem:getample_a} does not hold. Therefore, \ref{lem:getample}\ref{lem:getample_b} holds; that is, there are subsets $X'\subseteq \poi_{f_{\ref{lem:getample2}}}$ and $Y' \subseteq \poi_{g_{\ref{lem:getample2}}}$ such that: 
\begin{itemize}
    \item $|X'| = g_{\ref{lem:getample}}(d, l, \max\{r,3\}, r', s)$; 
    \item $|Y'| = f_{\ref{lem:getample}}(d, l, \max\{r,3\}, r', s)$; and
    \item $M' = (A_i: i \in X'; B_j:j \in Y')$ is $d$-$A$-ample. 
\end{itemize}
In particular, the third bullet above implies that $M'^T$ is $d$-$B$-ample. We again apply Lemma~\ref{lem:getample}, this time to $\mf{c}(M'^T)$. Since $\mf{c}(M'^T)$ is isomorphic to an induced subgraph of $\mf{c}(M^T)$, it follows from \eqref{st:ktt} and \eqref{st:ktt2} that the outcome \ref{lem:getample}\ref{lem:getample_a} does not hold. So \ref{lem:getample}\ref{lem:getample_b} holds; that is, there are subsets $X \subseteq X'$ and $Y \subseteq Y'$ such that: 
\begin{itemize}
    \item $|X| = l$ and $|Y| = s$; and
    \item $M'' = (B_j: j \in Y; A_i:i \in X)$ is $d$-$A$-ample. 
\end{itemize}
Furthermore, since $M'^T$ is $d$-$B$-ample, so is $M''$. Hence, $M''$ is $d$-ample. But now $M''^T$ satisfies the outcome \ref{lem:getample2}\ref{lem:getample2_c}. This completes the proof of Lemma~\ref{lem:getample2}.
\end{proof}

\begin{theorem} \label{thm:mainktt}
    For every $r \in \poi$, there is a constant $f_{\ref{thm:mainktt}}=f_{\ref{thm:mainktt}}(r)\in \poi$ with the following property. Let $G$ be a graph and assume that there is a $2$-ample linear induced $(f_{\ref{thm:mainktt}},f_{\ref{thm:mainktt}})$-model in $G$. Then $G$ has an induced minor isomorphic to $W_{r \times r}$. 
\end{theorem}
\begin{proof}
    Let $f_{\ref{thm:mainktt}}=f_{\ref{thm:mainktt}}(r) = 2 f_{\ref{thm:wallminor}}(r) + 1$.

    Let $M = (A_1, \dots, A_{f_{\ref{thm:mainktt}}}; B_1, \dots, B_{f_{\ref{thm:mainktt}}})$ be a $2$-ample linear induced $(f_{\ref{thm:mainktt}},f_{\ref{thm:mainktt}})$-model in $G$. Let $G_M=G[A(M)\cup B(M)]$. For every subset $X\subseteq A(M)\cup B(M)$, we define
    $$I_X = \{i \in \poi_{f_{\ref{thm:mainktt}}} : X \cap A_i \neq \varnothing\}$$
    and 
    $$J_X = \{j \in \poi_{f_{\ref{thm:mainktt}}} : X \cap B_j \neq \varnothing\}.$$

Let $F$ be the set of all edges of $G_M$ with an end in $A(M)$ and an end in $B(M)$. Let $Z$ be the set of all vertices in $G_M$ that are not incident with any edge in $F$. Since $M$ is linear, it follows that every vertex in $Z$ has degree at most two in $G_M$.

Let $H$ be the spanning subgraph of $G_M$ with edge set $F$. It follows that $H$ is bipartite with bipartition $(A(M), B(M))$, and that $Z$ is the set of all isolated vertices of $H$. Let $\mca{K}$ be the set of all components of $G_M\setminus Z$. Since every vertex in $(A(M)\cup B(M))\setminus Z$ is incident with an edge in $F$, it follows that for every $K\in \mca{K}$, we have $|I_K|,|J_K|\geq 1$. But we can indeed prove the following: 

\sta{\label{st:anti} For every $K\in \mca{K}$, we have $|I_K|=|J_K|=1$.}

Suppose not. Let $P=p_1\dd \cdots\dd p_k$ be a shortest path in $K$ such that $|I_P| > 1$ or $|J_P| > 1$. By symmetry, we may assume that $|I_P| > 1$. By the minimality of $P$, there exist distinct $i, i' \in I_P$ and $j \in J_P$ such that $p_1 \in A_i$ and $p_k \in A_{i'}$, and $p_2, \dots, p_{k-1} \in B_j$. Since $M$ is 2-ample, it follows that $k \geq 5$, and furthermore, $I(N(p_3)) \subseteq \{i\}$. If $p_3$ has a neighbor $p_1' \in A_i$, then $p_1' \in K$ and $p_1' \dd p_3 \dd \cdots \dd p_k$ contradicts the minimality of $P$. It follows that $p_3$ has no neighbors in $A(M)$, and hence $p_3 \in Z$, contrary to the assumption that $p_3 \in K$. This proves \eqref{st:anti}. 
\medskip
  
Now, let $G'$ be the graph obtained from $G_M$ by contracting each set $K\in \mca{K}$ to a vertex $v_K$. Then $G'$ is an induced minor of $G_M$ (and so of $G$) with $V(G')=\{v_K : K\in \mca{K}\}\cup Z$. It also follows that:

\sta{\label{st:deg3} No two vertices of degree at least three in $G'$ are adjacent.}

By construction $\{v_K : K \in \mca{K}\}$ is a stable set in $G'$, and the vertices in $Z$ are of degree at most two in $G'$ because they are of degree at most two in $G_M$. This proves \eqref{st:deg3}.
\medskip

Furthermore, we deduce that $G'$ has large treewidth:

\sta{\label{st:treewidth} $G'$ has treewidth at least $(f_{\ref{thm:mainktt}}-1)/2 = f_{\ref{thm:wallminor}}(r)$.}

    Let $G''$ arise from $G'$ by replacing each vertex $v_K$ for $K\in \mca{K}$ with two adjacent copies $v^A_K$ and $v^B_K$. It is straightforward to observe that $\tw(G'') \leq 2 \tw(G')+1$.
    Therefore, it suffices to show that $\tw(G'') \geq f_{\ref{thm:mainktt}}$. To that end, for every $i \in \poi_{f_{\ref{thm:mainktt}}}$, let $$A_i' = \{v_K^A: K\in \mca{K}, i\in I_K\}\cup (A_i\cap Z)$$
    and let
    $$B_i' = \{v_K^B: K\in\mca{K}, i\in J_K\}\cup (B_i\cap Z).$$
    It follows that:
    \begin{itemize}
        \item The sets $(A_i': i \in \poi_{f_{\ref{thm:mainktt}}})$ are connected. This is because $A_i$ is connected and $A_i'$ arose from $A_i$ by identifying the vertices in $A_i\cap K$ for each $K \in \mca{K}$. Likewise, the sets $(B_i': i \in \poi_{f_{\ref{thm:mainktt}}})$ are connected. 
        \item The sets $(A_i':i \in \poi_{f_{\ref{thm:mainktt}}})$ are pairwise disjoint. This is due to \eqref{st:anti} and the fact that the sets $(A_i:i \in \poi_{f_{\ref{thm:mainktt}}})$ are pairwise disjoint. Likewise, the sets $(B_i':i \in \poi_{f_{\ref{thm:mainktt}}})$ are pairwise disjoint. 
    
    \item For all $i, j \in \poi_{f_{\ref{thm:mainktt}}}$, the sets $A_i'$ and $B_j'$ are disjoint (by construction) and not anticomplete in $G$. The latter holds because $G$ has an edge $ab\in F$ with $a \in A_i$ and $b \in B_i$ (as $M$ is a induced $(f_{\ref{thm:mainktt}},f_{\ref{thm:mainktt}})$-model), and so there is a component $K\in \mca{K}$ with $a,b\in K$. But now $i\in I_K$ and $j\in J_K$, and $v_K^A \in A_i'$ and $v_K^B \in B_j'$ are adjacent.
    \end{itemize}
    
    From the above three bullets, we conclude that $G''$ has a minor isomorphic to $K_{f_{\ref{thm:mainktt}}, f_{\ref{thm:mainktt}}}$; thus $\tw(G'') \geq f_{\ref{thm:mainktt}}$. This proves \eqref{st:treewidth}. 

    \medskip

    From \eqref{st:treewidth} and Theorem \ref{thm:wallminor}, it follows that $G'$ has a subgraph isomorphic to a subdivision of $W_{r \times r}$. Let us denote this subgraph by $W$. We claim that:
    
    \sta{\label{st:Wisinduced}  $W$ is an induced subgraph of $G'$.}
    
   Suppose not; let $xy \in E(G')$ with $x, y \in W$ such that $x$ and $y$ are not adjacent in $W$. Since every vertex of $W$ has degree at least two in $W$, it follows that $x, y$ are adjacent vertices of degree at least three in $G'$, contrary to \eqref{st:deg3}. This proves \eqref{st:Wisinduced}.
   \medskip
   
   By contracting some edges of $W$ if needed, we obtain an induced minor of $G'$ isomorphic to $W_{r\times r}$. Since $G'$ is an induced minor of $G$, we deduce that $G$ has an induced minor isomorphic to $W_{r\times r}$. This completes the proof of Theorem~\ref{thm:mainktt}.
\end{proof}

Now Theorem \ref{thm:main_ind_minor} follows from combining Theorem~\ref{thm:2linear}, Lemma~\ref{lem:getample2}, and Theorem~\ref{thm:mainktt}. 

\section{Acknowledgements}
 This work was partly done during the 2024 Barbados Graph Theory Workshop at the Bellairs Research Institute of McGill University, in Holetown, Barbados. We thank the organizers for inviting us and for creating a stimulating work environment. We also thank Bogdan Alecu who worked with us on the results in Section~\ref{sec:zig}.
 
\bibliographystyle{plain}
\bibliography{ref}

@article {multiramsey,
    AUTHOR = {Ramsey, Frank P.},
     TITLE = {On a {P}roblem of {F}ormal {L}ogic},
   JOURNAL = {Proc. London Math. Soc. (2)},
  FJOURNAL = {Proceedings of the London Mathematical Society. Second Series},
    VOLUME = {30},
      YEAR = {1929},
    NUMBER = {4},
     PAGES = {264--286},
      ISSN = {0024-6115},
   MRCLASS = {99-04},
  MRNUMBER = {1576401},
       DOI = {10.1112/plms/s2-30.1.264},
       URL = {https://doi.org/10.1112/plms/s2-30.1.264},
}

@article{tw11,
  title={{I}nduced subgraphs and tree decompositions {XI}. {L}ocal structure in even-hole-free graphs of large treewidth},
  author={Alecu, Bogdan and Chudnovsky, Maria and Hajebi, Sepehr and Spirkl, Sophie},
  journal={{\rm Manuscript available at \url{https://arxiv.org/abs/2309.04390}}},
  year={2023}
}

@article{tw12,
  title={{I}nduced subgraphs and tree decompositions {XII}. {G}rid Theorem for pinched graphs},
  author={Alecu, Bogdan and Chudnovsky, Maria and Hajebi, Sepehr and Spirkl, Sophie},
  journal={{\rm Manuscript available at \url{https://arxiv.org/abs/2309.12227}}},
  year={2023}
}

@article {tw13,
    AUTHOR = {Alecu, Bogdan and Chudnovsky, Maria and Hajebi, Sepehr and
              Sprikl, Sophie},
     TITLE = {Induced subgraphs and tree decompositions {XIII}. {B}asic
              obstructions in {$\mathcal{H}$}-free graphs for finite {$\mathcal{H}$}},
   JOURNAL = {Adv. Comb.},
  FJOURNAL = {Advances in Combinatorics},
      YEAR = {2024},
     PAGES = {Paper No. 6, 30pp},
      ISSN = {2517-5599},
   MRCLASS = {05C75 (05C83)},
  MRNUMBER = {4833525},
}

@article{chordalehf,
  title={{C}hordal graphs, even-hole-free graphs and sparse obstructions to bounded treewidth},
  author={Hajebi, Sepehr},
  journal={{\rm Manuscript available at \url{https://arxiv.org/abs/2401.01299}}},
  year={2023}
}

@unpublished{davies,
    author = {James Davies},
    note = {Appeared in an Oberwolfach technical report, {\em DOI:10.4171/OWR/2022/1}}
}

@article{sintiari2021theta,
  title={({{T}}heta, triangle)-free and (even hole, {{$K_4$}})-free graphs—{{P}}art 1: {{L}}ayered wheels},
  author={Sintiari, Ni Luh Dewi and Trotignon, Nicolas},
  journal={Journal of Graph Theory},
  volume={97},
  number={4},
  pages={475--509},
  year={2021},
  publisher={Wiley Online Library}
}

@article {deathstar,
    AUTHOR = {Bonamy, Marthe and Bonnet, \'Edouard and D\'epr\'es, Hugues
              and Esperet, Louis and Geniet, Colin and Hilaire, Claire and
              Thomass\'e, St\'ephan and Wesolek, Alexandra},
     TITLE = {Sparse graphs with bounded induced cycle packing number have
              logarithmic treewidth},
   JOURNAL = {J. Combin. Theory Ser. B},
  FJOURNAL = {Journal of Combinatorial Theory. Series B},
    VOLUME = {167},
      YEAR = {2024},
     PAGES = {215--249},
      ISSN = {0095-8956,1096-0902},
   MRCLASS = {05C75 (05C70 05C85)},
  MRNUMBER = {4723425},
MRREVIEWER = {Keith\ J.\ Edwards},
       DOI = {10.1016/j.jctb.2024.03.003},
       URL = {https://doi.org/10.1016/j.jctb.2024.03.003},
}

@unpublished{pohoata2014unavoidable,
    title={Unavoidable induced subgraphs of large graphs},
  author={Pohoata, Andrei Cosmin},
  note={Senior thesis, Princeton University},
  year={2014}
}

@article{GMV,
  title={Graph minors. {{V}}. {{E}}xcluding a planar graph},
  author={Robertson, Neil and Seymour, Paul},
  journal={Journal of Combinatorial Theory, Series B},
  volume={41},
  number={1},
  pages={92--114},
  year={1986},
  publisher={Elsevier}
}

@article{aboulker,
  title={On the tree-width of even-hole-free graphs},
  author={Aboulker, Pierre and Adler, Isolde and Kim, Eun Jung and Sintiari, Ni Luh Dewi and Trotignon, Nicolas},
  journal={European Journal of Combinatorics},
  volume={98},
  pages={103394},
  year={2021},
  publisher={Elsevier}
}

@article {pinned,
    AUTHOR = {Hajebi, Sepehr},
     TITLE = {Induced subdivisions with pinned branch vertices},
   JOURNAL = {European J. Combin.},
  FJOURNAL = {European Journal of Combinatorics},
    VOLUME = {124},
      YEAR = {2025},
     PAGES = {Paper No. 104072},
      ISSN = {0195-6698,1095-9971},
   MRCLASS = {05C75 (05C83)},
  MRNUMBER = {4798868},
       DOI = {10.1016/j.ejc.2024.104072},
       URL = {https://doi.org/10.1016/j.ejc.2024.104072},
}

@article{tw9,
  title={{I}nduced subgraphs and tree decompositions {IX}. {G}rid Theorem for perforated graphs},
  author={Alecu, Bogdan and Chudnovsky, Maria and Hajebi, Sepehr and Spirkl, Sophie},
  journal={{\rm Manuscript available at \url{https://arxiv.org/abs/2305.15615}}},
  year={2023}
}

@article {hglemma,
    AUTHOR = {Ding, Guoli and Seymour, Paul and Winkler, Peter},
     TITLE = {Bounding the vertex cover number of a hypergraph},
   JOURNAL = {Combinatorica},
  FJOURNAL = {Combinatorica. An International Journal on Combinatorics and
              the Theory of Computing},
    VOLUME = {14},
      YEAR = {1994},
    NUMBER = {1},
     PAGES = {23--34},
      ISSN = {0209-9683},
   MRCLASS = {05C65 (05C35 05C70)},
  MRNUMBER = {1273198},
MRREVIEWER = {Ioan\ Tomescu},
       DOI = {10.1007/BF01305948},
       URL = {https://doi.org/10.1007/BF01305948},
}

@MISC{wollananswers,
    TITLE = {No big clique minor but a big grid minor},
    AUTHOR = {Wollan, Paul},
    HOWPUBLISHED = {Posted on {\em Mathoverflow}},
    NOTE = {\url{https://mathoverflow.net/q/145201} (version: 2013-10-18)},
    EPRINT = {https://mathoverflow.net/q/145201},
    URL = {https://mathoverflow.net/q/145201}
}

@MISC{seymourtalk,
    TITLE = {{F}ive-minute talk},
    AUTHOR = {Seymour, Paul},
    HOWPUBLISHED = {{BIRS} {W}orkshop on {G}eometric and {T}opological {G}raph {T}heory (2013-09-30)},
    NOTE = {Recording available at \url{http://www.birs.ca/events/2013/5-day-workshops/13w5091/videos/watch/201309301509-5minutetalks.html} (see 28:35).},
    EPRINT = {},
    URL = {}
}

@book {diestel,
    AUTHOR = {Diestel, Reinhard},
     TITLE = {Graph theory},
    SERIES = {Graduate Texts in Mathematics},
    VOLUME = {173},
   EDITION = {Fifth},
   EDITION = {Paperback},
 PUBLISHER = {Springer, Berlin},
      YEAR = {2018},
     PAGES = {xviii+428},
      ISBN = {978-3-662-57560-4; 978-3-662-53621-6},
   MRCLASS = {05-01 (01A75 05Cxx)},
  MRNUMBER = {3822066},
}

@article {dvorak,
    AUTHOR = {Dvo\v{r}\'{a}k, Zden\v{e}k},
     TITLE = {Induced subdivisions and bounded expansion},
   JOURNAL = {European J. Combin.},
  FJOURNAL = {European Journal of Combinatorics},
    VOLUME = {69},
      YEAR = {2018},
     PAGES = {143--148},
      ISSN = {0195-6698,1095-9971},
   MRCLASS = {05C10},
  MRNUMBER = {3738146},
       DOI = {10.1016/j.ejc.2017.10.004},
       URL = {https://doi.org/10.1016/j.ejc.2017.10.004},
}

@article{nicolassuggests,
  title={{U}navoidable induced subgraphs in graphs with complete bipartite induced minors},
  author={Chudnovsky, Maria and Hatzel, Meike and Korhonen, Tukka and Trotignon, Nicolas and Wiederrecht, Sebastian},
  journal={{\rm Manuscript available at \url{https://arxiv.org/abs/2405.01879}}},
  year={2024}
}

@article {GM1,
    AUTHOR = {Robertson, Neil and Seymour, P. D.},
     TITLE = {Graph minors. {I}. {E}xcluding a forest},
   JOURNAL = {J. Combin. Theory Ser. B},
  FJOURNAL = {Journal of Combinatorial Theory. Series B},
    VOLUME = {35},
      YEAR = {1983},
    NUMBER = {1},
     PAGES = {39--61},
      ISSN = {0095-8956,1096-0902},
   MRCLASS = {05C35},
  MRNUMBER = {723569},
MRREVIEWER = {Y.\ Roditty},
       DOI = {10.1016/0095-8956(83)90079-5},
       URL = {https://doi.org/10.1016/0095-8956(83)90079-5},
}

@article{tw18,
  title={{I}nduced subgraphs and tree decompositions {XVIII}. {O}bstructions to bounded pathwidth},
  author={Chudnovsky, Maria and Hajebi, Sepehr and Spirkl, Sophie},
 journal={{\rm Manuscript available at \url{https://arxiv.org/abs/2412.17756}}},
  year={2024}
}

@article{approxpw,
author = {Groenland, Carla and Joret, Gwena\"{e}l and Nadara, Wojciech and Walczak, Bartosz},
title = {Approximating Pathwidth for Graphs of Small Treewidth},
year = {2023},
issue_date = {2023},
publisher = {Association for Computing Machinery},
address = {New York, NY, USA},
volume = {19},
number = {2},
issn = {1549-6325},
url = {https://doi.org/10.1145/3576044},
doi = {10.1145/3576044},
journal = {ACM Trans. Algorithms},
month = mar,
articleno = {16},
numpages = {19}
}

\end{document}